\documentclass[11pt,letterpaper,reqno]{amsart}
\usepackage{graphicx}
\usepackage{amsmath}
\usepackage{amssymb}
\usepackage{amsfonts}
\usepackage{amsrefs}
\usepackage{enumerate}
\usepackage[mathscr]{eucal}

\usepackage[usenames]{color}

\newcommand{\R}{\mathbb{R}}
\newcommand{\ol}{\overline}

\newtheorem{thm}{Theorem}
\newtheorem{lemma}[thm]{Lemma}
\newtheorem{prop}[thm]{Proposition}
\newtheorem{cor}[thm]{Corollary}

\newtheorem{assumption}[thm]{Assumption}

\theoremstyle{definition}
\newtheorem{definition}[thm]{Definition}

\theoremstyle{remark}
\newtheorem{remark}{Remark}
\newtheorem*{ack}{Acknowledgments}
\newtheorem*{outline}{Outline}

\DeclareMathOperator{\interior}{int}
\DeclareMathOperator{\Div}{div}
\DeclareMathOperator{\spt}{spt}
\DeclareMathOperator{\sing}{sing}

\def\madm{m_{\mathrm{ADM}}}
\def\miso{m_{\mathrm{iso}}}
\def\mh{m_{\mathrm{Haw}}}

\begin{document}

\title[Lower semicontinuity of mass under $C^0$ convergence]{Lower semicontinuity of mass under $C^0$ convergence and Huisken's isoperimetric mass}
\author{Jeffrey L. Jauregui}
\address{Dept. of Mathematics,
Union College, 807 Union St.,
Schenectady, NY 12308}
\email{jaureguj@union.edu}
\author{Dan A. Lee}
\address{Graduate Center and Queens College, City University of New York, 365 Fifth Avenue,
New York, NY 10016, USA}
\email{dan.lee@qc.cuny.edu}

\begin{abstract}
Given a sequence of asymptotically flat 3-manifolds of nonnegative scalar curvature with outermost minimal boundary, converging in the pointed $C^0$ Cheeger--Gromov sense to an asymptotically flat limit space, we show that the total mass of the limit is bounded above by the liminf of the total masses of the sequence. In other words, total mass is lower semicontinuous under such convergence. In order to prove this, we use Huisken's isoperimetric mass concept, together with a modified weak mean curvature flow argument. We include a brief discussion of Huisken's work before explaining our extension of that work. The results are all specific to three dimensions.
\end{abstract}

\maketitle

\section{Introduction}
In \cite{Jauregui}, the first-named author studied how the ADM mass behaves under geometric convergence of a sequence of asymptotically flat 3-manifolds, proving the following.

\begin{thm}[Lower semicontinuity of mass under $C^2$ convergence \cite{Jauregui}]\label{thm:Jauregui}
Let $(M_i, g_i, p_i)$ be a sequence of pointed smooth asymptotically flat\footnote{We define smooth asymptotic flatness for (one-ended) $3$-manifolds (possibly with boundary) in the usual way: outside a compact set, there is a coordinate chart for $M$ in which the smooth Riemannian metric $g$ satisfies $g_{ij}-\delta_{ij}=O_2(|x|^{-p})$ for some $p>\tfrac{1}{2}$. Here, $O_2$ means that the first two derivatives have appropriate decay. We also require that the scalar curvature is integrable so that the ADM mass is well-defined. Note that $g$ is automatically complete.} 3-manifolds without boundary, such that each $(M_i ,g_i)$ has nonnegative scalar curvature and contains no compact minimal surfaces.
Then if $(M_i, g_i, p_i)$ converges in the pointed $C^2$ Cheeger--Gromov sense to some pointed $C^2$ asymptotically flat 3-manifold $(N, h, q)$, then 
\[ \madm(N, h) \le \liminf_{i\to\infty} \madm(M_i, g_i).\]
\end{thm} 

As pointed out in \cite{Jauregui}, such a result is not true without the hypothesis of nonnegative scalar curvature (or without the hypothesis of no compact minimal surfaces). Therefore this semicontinuity property is intimately connected to scalar curvature. In fact, the proof of Theorem \ref{thm:Jauregui} relies on  G. Huisken and T. Ilmanen's inverse mean curvature flow proof of the Riemannian Penrose equality \cite{Huisken-Ilmanen:2001} --- an argument that also suffices to prove the positive mass theorem \cites{Schoen-Yau:1979, Witten:1981}.  Indeed, it was also observed in \cite{Jauregui} that the positive mass theorem itself is an immediate consequence of Theorem \ref{thm:Jauregui}.

The purpose of the present work is to generalize Theorem \ref{thm:Jauregui} to allow for only $C^0$ Cheeger--Gromov convergence. Aside from independent interest, such a generalization is an important step forward if one wants to attack the Bartnik ``minimal mass extension'' conjecture \cites{Bartnik:1997,Bartnik:2002} using the direct method (see \cite{Jauregui}). Using only $C^0$ convergence as a hypothesis is significantly more difficult, because it essentially means that one cannot make any quantitative use of curvature. We find it natural to use Huisken's isoperimetric mass and quasilocal isoperimetric mass, since both only use $C^0$ data of the metric (area and volume, to be precise). Our main result is the following (which also generalizes Theorem \ref{thm:Jauregui} by allowing a minimal boundary). 

\begin{thm}[Lower semicontinuity of mass under $C^0$ convergence]\label{thm:main}
Let $(M_i, g_i, p_i)$ be a sequence of pointed smooth asymptotically flat 3-manifolds whose boundaries are empty or minimal, such that each 
$(M_i ,g_i)$ has nonnegative scalar curvature and contains no compact minimal surfaces in its interior. Then if $(M_i, g_i, p_i)$ converges in the pointed $C^0$ Cheeger--Gromov sense to some pointed $C^0$ asymptotically flat 3-manifold $(N, h, q)$, then 
\[ \miso(N, h) \le \liminf_{i\to\infty} \madm(M_i, g_i),\]
where $\miso$ is Huisken's isoperimetric mass.
\end{thm}
The definitions of isoperimetric mass, $C^0$ asymptotic flatness, and $C^0$ Cheeger--Gromov convergence appear in the next section.

Note that this theorem allows for the possibility that the limit metric $h$ is not smooth. However, if $h$ happens to be smooth\footnote{In this case $h$ itself has nonnegative scalar curvature, since it is a $C^0$ limit of smooth metrics of nonnegative scalar curvature  \cites{Gromov:2014, Bamler:2015}.} and asymptotically flat, then $\miso$ can be replaced by $\madm$ in Theorem \ref{thm:main}, thanks to the following theorem.

\begin{thm}\label{thm:iso-adm}
If $(M,g)$ is a smooth asymptotically flat 3-manifold and has nonnegative scalar curvature, then 
\[ \miso(M,g)=\madm(M,g).\]
\end{thm}
Huisken announced this theorem when he first introduced his isoperimetric mass concept (see, \textit{e.g.}, \cite{Huisken:Morse}). P.\ Miao observed that $\miso(M,g)\ge \madm(M,g)$ follows from volume estimates of X.-Q.\ Fan, Y.\ Shi, and L.-F.\ Tam \cite{Fan-Shi-Tam:2009}\footnote{Miao's result showed equality, but for a different definition of the isoperimetric mass that restricts to coordinate balls.}.  A proof of the reverse inequality, following Huisken's approach, appears in this paper. 
In fact, we prove a stronger statement, Theorem \ref{thm:iso-adm2}, which is the main idea driving the proof of Theorem \ref{thm:main}. 

\begin{outline} Section \ref{sec:def} contains the basic definitions used in the paper. Section \ref{sec:isop} presents ideas of Huisken on how the isoperimetric profile function $\phi_m$ of Schwarzschild space motivates his definition of isoperimetric mass. It also includes a proof of Huisken's monotonicity result for $\phi_m$ under mean curvature flow. In Section \ref{sec:outline} we outline our broad approach to proving Theorem \ref{thm:main}, which naturally leads to a discussion of the key difficulties, including the issue of a mean curvature flow becoming singular and/or disconnected. In Section \ref{sec:convexity} we prove some additional properties of $\phi_m$ which are used later to show that Huisken's monotonicity result holds even for a disconnected flow. Our argument requires a weak version of mean curvature flow, which must be further modified to freeze any connected components of sufficiently small perimeter; this formalism is discussed in Sections \ref{sec:weak_flow} and \ref{sec:mod_flow_prop}. In the final two sections we prove a result of independent interest (Theorem \ref{thm:iso-adm2}), and finally we use this result to prove our main result (Theorem \ref{thm:main}).
\end{outline}

\begin{ack}
The authors thank Gerhard Huisken and Felix Schulze for helpful discussions.
\end{ack}

\section{Some definitions and notation}
\label{sec:def}

\begin{definition} A pair $(M,g)$ is a  \emph{$C^0$ asymptotically flat 3-manifold} if $M$ is a smooth 3-manifold (possibly with boundary), $g$ is a continuous Riemannian metric on $M$, and the following property holds: there exists a compact set $K \subset M$ and a diffeomorphism $\Phi: M \smallsetminus K \longrightarrow \R^3 \smallsetminus B$, for a closed ball $B$, such that in the coordinates $(x^i)$ determined by $\Phi$,
$$|g_{ij} - \delta_{ij}| = O(|x|^{-p}),$$
for some constant $p > 0$, where $|x| = \sqrt{(x^1)^2 + (x^2)^2 + (x^3)^2}$.
\end{definition}

Note that continuous metrics allow one to define Hausdorff $k$-measure (and in particular volumes and perimeters of regions) in the usual way. We use the notation $|\cdot |_g$ to denote both the area (2-dimensional Hausdorff measure) and volume (3-dimensional Hausdorff measure), with the meaning to be inferred from the context. When the metric is understood, we may omit the subscript. 

For now, let us assume that $M$ has no boundary.  Given a bounded open set $\Omega$ in $M$, we define $\partial \Omega$ to be its topological boundary, and $\partial^* \Omega \subset \partial \Omega$ to be its \emph{reduced boundary}\footnote{See \cite{Ambrosio-et-al:2000}, for instance,  for definitions of reduced boundary and perimeter.}. If $\Omega$ is mildly regular, \textit{e.g.} $C^1$, then $\partial^* \Omega = \partial \Omega$.  We will typically assume that $\Omega$ has finite perimeter, so that $|\partial^* \Omega|_g$ is the \emph{perimeter} of $\Omega$. 
We say that $\Omega$ is \emph{outward-minimizing} if it minimizes perimeter among all bounded open sets $\Omega' \supset \Omega$ of finite perimeter in $M$, that is, if
$|\partial^* \Omega'|_g \geq |\partial^* \Omega|_g$
for all such $\Omega'$. 

\begin{thm}[Regularity Theorem 1.3(iii) of \cite{Huisken-Ilmanen:2001}]
\label{regularity-thm}
Let $\Omega$ be a bounded open subset of a smooth asymptotically flat 3-manifold $(M,g)$ without boundary such that  $\partial \Omega$ is $C^2$. Then there
 exists an outward-minimizing open set $\tilde \Omega \supset \Omega$ such that $\partial \tilde\Omega$ is $C^{1,1}$ and $\tilde\Omega$ has the least possible perimeter among all regions containing $\Omega$. In fact, $\tilde \Omega$ can be taken as the intersection of all outward-minimizing sets that contain $\Omega$. Moreover, $\partial \tilde \Omega \smallsetminus \partial \Omega$, if nonempty, is a smooth minimal surface.
\end{thm}

Such $\tilde \Omega$ is called the \emph{minimizing hull} of $\Omega$.  Theorem \ref{regularity-thm} requires smoothness, but we will only need to take minimizing hulls with respect to smooth metrics.

For technical reasons it will be convenient to work in manifolds without boundary. If $(M^3,g)$ is a $C^0$ asymptotically flat manifold with  boundary, we can ``fill in'' the components of $\partial M$ by gluing smooth, connected 3-manifolds $W_1, \ldots, W_n$ along their boundaries to the components of $\partial M$,  in such a way that $M^{\mathrm{fill}} = M \cup W_1 \cup \ldots \cup W_n$ is a smooth, connected manifold without boundary, and $W= W_1 \cup \ldots \cup W_n$ is an open set in $M^{\mathrm{fill}}$ with compact closure.  We also extend the metric $g$ continuously to all of $M^{\mathrm{fill}}$ (and smoothly if $g$ is smooth on $M$).  We refer to $W$ as the \emph{fill-in region} and $M^{\mathrm{fill}}$ as the \emph{filled-in} manifold of $M$. (If $\partial M$ is empty, let $M^{\mathrm{fill}} = M$.) Note that this is the same procedure used in \cite[Section 6]{Huisken-Ilmanen:2001}, and that the choice of fill-in is not important.

\begin{definition} 
A bounded open set of finite perimeter $\Omega$ in $M^{\mathrm{fill}}$ is called an \emph{allowable region in $M$} if each component of the fill-in region $W$ is either contained in $\Omega$ or disjoint from $\Omega$.
\end{definition}
In particular, if $\Omega$ is allowable in $M$, then $\partial \Omega$ is a subset of $M$. Note that the minimizing hull of an allowable region need not be allowable. For an illustration, see Figure \ref{fig:allowable}.
\begin{figure}[ht]
\begin{center}
\includegraphics[scale=0.60]{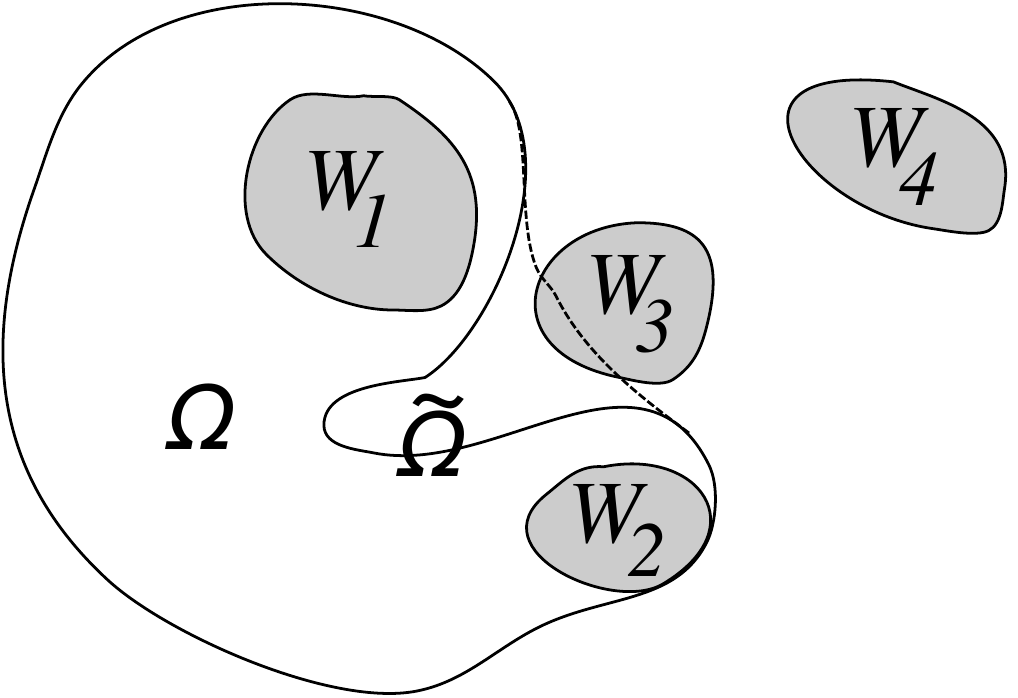}
\caption{\small Above, $W_1$, $W_2$, $W_3$, and $W_4$ represent the fill-in regions, whose boundaries comprise $\partial M$. The open set $\Omega$ is allowable, because it contains or is disjoint from each of the fill-in regions. The minimizing hull, $\tilde \Omega$, is not an allowable region, because it cuts through $W_3$.
 \label{fig:allowable}}
\end{center}
\end{figure}

We define the volume $|\Omega|$ of an allowable region $\Omega$ to be the Hausdorff 3-measure in $M$ of $\Omega \smallsetminus W$, so as to neglect to contribution from the fill-in components (and make the definition independent of the choice of fill-in). The \emph{isoperimetric ratio} of an allowable region $\Omega$ is $\frac{|\partial^* \Omega|^{3/2}}{ |\Omega|}$.  The \emph{isoperimetric constant of a manifold}, $c=c(M,g)$, is the infimum of the isoperimetric ratios over all allowable regions in $M$. For instance, $\R^3$ has isoperimetric constant $6\sqrt{\pi}$. In general, the isoperimetric inequality states that $|\partial^* \Omega|^{3/2} \geq c |\Omega|$ for every allowable region $\Omega$.

Some facts immediately follow from the definition of $C^0$ asymptotic flatness.
\begin{lemma}
\label{lemma_AF_asymptotics}
Suppose $(M,g)$ is a $C^0$ asymptotically flat 3-manifold. Let $S_r$ denote the coordinate sphere in $M$ defined by $|x|=r$, and let $B_r$ be the bounded open set
with $\partial B_r = S_r$. Then
\begin{enumerate}
\item $\displaystyle\lim_{r \to \infty} \frac{|S_r|_g}{4\pi r^2} = 1$.
\item $\displaystyle\lim_{r \to \infty} \frac{|S_r|_g^{3/2}}{|B_r|_g} = 6 \sqrt{\pi}$.
\end{enumerate}
\end{lemma}

\begin{lemma}\label{sop-lower-bound}
The isoperimetric constant of a $C^0$ asymptotically flat 3-manifold $(M,g)$ is strictly positive.
\end{lemma}
\begin{proof}
Let $(M^3,g)$ be $C^0$ asymptotically flat, and extend $g$ to the filled-in manifold $M^{\mathrm{fill}}$. It suffices to show that $(M^{\mathrm{fill}},g)$ has positive isoperimetric constant. Since $g$ is uniformly equivalent to a smooth metric that is Euclidean outside of a compact set, this can be proved using isoperimetric estimates in \cite{Croke:1980}, for example.
\end{proof}

We conclude this section by recalling the definition of $C^0$ pointed Cheeger--Gromov convergence. In the following, a $C^0$ Riemannian manifold refers to a smooth manifold (possibly with boundary) equipped with a continuous Riemannian metric.
\begin{definition}
\label{def:CG}
A sequence of complete pointed $C^0$ Riemannian $n$-manifolds $(M_i,g_i,p_i)$  converges to a complete pointed $C^0$ Riemannian $n$-manifold $(N,h,q)$ \emph{in the $C^0$ pointed Cheeger--Gromov sense} if, given any $R> 0$, there exists an open set $U$ in $N$ containing the ball $B_h(q,R)$ and smooth embeddings $\Phi_i : U \longrightarrow M_i$ (for all $i$ sufficiently large) whose image contains the ball $B_{g_i}(p_i,R)$ in $M_i$, such that $\Phi_i^* g_i$ converges in $C^0$ (\textit{i.e.}, uniformly) to $h$ on $U$, as $i \to \infty$.
\end{definition}

\section{Huisken's isoperimetric mass}
\label{sec:isop}

Most of the content of this section is based on material that Huisken has either presented in public lectures \cites{Huisken:Morse, Huisken:2006} or communicated to us personally \cite{Huisken:private}. We are grateful to him for allowing us to describe some of his unpublished work here. 

Recall that we can choose coordinates $x=(x^1, x^2, x^3)$ for the Schwarzschild metric $g$ of mass $m>0$ so that 
\begin{equation}\label{coordinates}
g_{ij}(x) = \left( 1+ \frac{m}{2r}\right)^4\delta_{ij},
\end{equation}
where $r=|x| \geq \frac{m}{2}$. In this paper we consider the totally geodesic sphere $r=\frac{m}{2}$, often called the horizon, to be the boundary of the Schwarzschild manifold. Note that the horizon has area $16\pi m^2$.

Consider an $r=$ constant sphere of area $A$ in a Schwarzschild space of mass $m>0$ enclosing volume $V$, where enclosed volume means the volume inside the given sphere and outside the horizon. This sets up a relationship between the quantities $A$, $V$, and $m$. Define $\phi_m(A)$ to be the volume $V$ as a function of $m$ and $A$, \textit{i.e.}\ the isoperimetric profile\footnote{H. Bray proved that the isoperimetric surfaces in the Schwarzschild space of mass $m>0$ are in fact the rotationally symmetric spheres \cite{Bray:1997}.} of the Schwarzschild space of mass $m$. It is defined whenever $A \ge 16\pi m^2$ and $m\ge 0$. (Note: the Schwarzschild manifold of mass zero is Euclidean 3-space.)
Obviously, $A\mapsto \phi_m(A)$ grows roughly like the Euclidean isoperimetric profile $\phi_0(A)=\frac{1}{6\sqrt{\pi}}A^{3/2}$ at the top order, but the next highest order behavior in $A$ turns out to be $\tfrac{1}{2}mA$, so that the asymptotics of the isoperimetric profile function ``see'' the mass. More precisely, we have the following.

\begin{lemma} 
\label{lemma_lim_phi_m}
\[  \frac{2}{A}\left(\phi_{m}(A) - \frac{1}{6\sqrt{\pi}}A^{3/2}\right) = m + O(A^{-1/2}),\]
where the error estimate $O(A^{-1/2})$ is uniform for $m \in [0,\mu]$, where $\mu> 0$ is any constant.
\end{lemma}

\begin{proof}
With the choice of coordinates (\ref{coordinates}) in the Schwarzschild space, the coordinate sphere of radius $r$ has area
\begin{equation}\label{area}
 A = 4\pi r^2 \left(1 + \frac{m}{2r}\right)^4,
 \end{equation}
 while the enclosed volume between the horizon at radius $m/2$ and the coordinate sphere of radius $r$ is
 \begin{equation}\label{volume}
 \phi_m(A)= \int_{m/2}^r 4\pi \rho^2  \left(1 + \frac{m}{2\rho}\right)^6 d\rho. 
 \end{equation} 
We estimate 
\[\left(1 + \frac{m}{2r}\right)^6 = 1 + \frac{3m}{r} + O(r^{-2}),\]
which is uniform in $m$ because $0 \leq m \leq \mu$. Now,
\begin{align*}
\frac{2}{A}\left(\phi_{m}(A) - \frac{1}{6\sqrt{\pi}}A^{3/2}\right) &= \frac{2}{A}\left[\int_{m/2}^r 4\pi \rho^2  \left(1 + \frac{m}{2\rho}\right)^6 d\rho -\frac{(4\pi)^{3/2}r^3}{6\sqrt{\pi}}\left(1 + \frac{m}{2r}\right)^6\right]\\
&= \frac{2}{A}\left[\int_{m/2}^r 4\pi \rho^2  \left(1 + \frac{3m}{\rho} + O(\rho^{-2})\right) d\rho\right.\\
&\qquad\qquad\qquad\qquad\qquad \left. -\frac{4\pi r^3}{3}\left(1 + \frac{3m}{r} + O(r^{-2})\right)\right]\\
&= \frac{2}{A}\left[\left( \frac{4\pi r^3}{3} + 6\pi mr^2 +O(r)\right) - \left(  \frac{4\pi r^3}{3} + 4\pi mr^2 +O(r)\right) \right]\\
&= \frac{2}{4\pi r^2 \left(1 + \frac{m}{2r}\right)^4}\left[2\pi m r^2 + O(r)\right]\\
&= m + O(r^{-1})\\
&= m+O(A^{-1/2})
\end{align*}
where we use the bound $m\le \mu$ to see that $O(A^{-1/2})$ is uniform among $m \in [0,\mu]$.
\end{proof}
This calculation motivates the following.
\begin{definition}[Huisken]
\label{def:isop_mass}
Let $(M,g)$ be a $C^0$ asymptotically flat 3-manifold. The \emph{quasilocal isoperimetric mass} of an allowable region
$\Omega$ is
\[ \miso(\Omega)=\miso(\Omega,g):=   \frac{2}{|\partial^* \Omega|}\left(|\Omega| - \frac{1}{6\sqrt{\pi}}|\partial^* \Omega|^{3/2} \right).\] 
The \emph{isoperimetric mass} of $(M,g)$ is defined by
\[\miso(M,g) = \sup_{\{\Omega_i\}_{i=1}^\infty} \left(\limsup_{i \to \infty} \miso(\Omega_i,g)\right),\]
where the supremum is taken over all exhaustions $\{\Omega_i\}_{i=1}^\infty$ of $M$ by allowable regions. Note that $\miso(M,g) \in [-\infty,\infty]$ is manifestly a geometric invariant.
\end{definition}
\begin{remark}
Other definitions of the isoperimetric mass could be used. For instance, instead of requiring exhaustions of $M$, one might instead consider using sequences of bounded open sets whose perimeters approach infinity. In the Appendix, we prove that these approaches are equivalent under the assumption of positive mass.
\end{remark}
One key feature of the quasilocal isoperimetric mass and the isoperimetric mass is that they only require notions of volume and area in order to make sense, and therefore they have great potential for applications to low regularity situations or, as is the case in this paper, low regularity convergence. On the other hand, the isoperimetric mass agrees with the usual definition of mass as described in Theorem \ref{thm:iso-adm}. The proof uses Hawking mass in an essential way.

\begin{definition}
Given a compact $C^2$ surface $\Sigma$ in a smooth Riemannian 3-manifold $(M, g)$, recall that the \emph{Hawking mass} of $\Sigma$ is
\begin{equation}\label{hawking}
 \mh(\Sigma) := \sqrt{\frac{|\Sigma|}{16\pi}} \left(1 - \frac{1}{16\pi} \int_\Sigma H^2\right). 
 \end{equation}
\end{definition}
The following important theorem (for which the Riemannian Penrose inequality is a corollary) is not explicitly stated in \cite{Huisken-Ilmanen:2001}, but it follows directly from the results there. 

\begin{thm}[Huisken--Ilmanen \cite{Huisken-Ilmanen:2001}]\label{thm:haw-adm}
Let $(M,g)$ be a smooth asymptotically flat 3-manifold whose boundary is empty or minimal, with nonnegative scalar curvature and no compact minimal surfaces in its interior. Then for any outward-minimizing allowable region with connected, $C^2$ boundary $\Sigma$, we have
\[ \mh(\Sigma) \leq \madm(M,g).\]
\end{thm}

We will also make use of the Riemannian Penrose inequality itself, stated below. In the case that $\partial M$ is connected, it follows immediately from Theorem \ref{thm:haw-adm}.
\begin{thm}[Bray \cite{Bray:2001}, cf. Huisken--Ilmanen \cite{Huisken-Ilmanen:2001}]\label{thm:RPI}
Let $(M,g)$ be a smooth asymptotically flat 3-manifold whose boundary is nonempty and minimal, with nonnegative scalar curvature and no compact minimal surfaces in its interior. Then
\[ \madm(M,g) \geq \sqrt{\frac{|\partial M|_g}{16\pi}}.\]
\end{thm}

We consider the proof of the inequality $\miso(M,g)\le\madm(M,g)$, which is the part of Theorem \ref{thm:iso-adm} that is most relevant to our work. The key ingredient is a monotonicity formula for mean curvature flow (MCF). A similar monotonocity had been used by F. Schulze for application to the isoperimetric inequality \cite{Schulze:2008}.

\begin{prop}[Huisken's relative volume monotonicity for MCF \cite{Huisken:Morse}] 
Let $(M,g)$ be a smooth asymptotically flat 3-manifold, and let $\{\Sigma_t\}_{t \in [0,T)}$ be a smooth mean curvature flow of compact surfaces, where $\Sigma_t = \partial \Omega_t$ for bounded open sets $\Omega_t \subset M$.  Let $m\ge0$ be a constant such that $|\Sigma_t|> 16\pi m^2$ and  $\mh(\Sigma_t)\le m$ for all $t\in[0,T)$.  Then
\[ \frac{d}{dt} \Big(\phi_m(|\Sigma_t|) - |\Omega_t|\Big) \leq 0.\]
\label{thm:huisken-monotonicity}
\end{prop}

\begin{remark}
Note that this result requires neither nonnegative scalar curvature nor connectedness of $\Sigma_t$. We will be interested in the case in which $m$ is the ADM mass of $(M,g)$.
\end{remark}

\begin{proof}
Let $A = A(t) = |\Sigma_t|> 16\pi m^2$.  Using the coordinates \eqref{coordinates},
 define $r=r(t)>\frac{m}{2}$ to be the radius of the unique coordinate sphere in the Schwarzschild manifold of mass $m$ whose area equals $A(t)$. 
Then
\begin{align*}
\frac{d}{dt} \Big(\phi_m(|\Sigma_t|) - |\Omega_t|\Big) &= \frac{d\phi_m}{dr} \frac{dr}{dA}\frac{dA}{dt} + \int_{\Sigma_t} H_t
\end{align*}
where $H_t$ is the mean curvature of $\Sigma_t$ (and equals the inward flow speed). From equation~\eqref{volume}, we have
\[\frac{d\phi_m}{dr} = 4\pi r^2\left(1+\frac{m}{2r}\right)^6,\]
and from equation \eqref{area}, we can derive
\[\frac{dA}{dr} = 8\pi r \left(1 + \frac{m}{2r}\right)^3\left(1 - \frac{m}{2r}\right).\]
Next, under mean curvature flow,
\[\frac{dA}{dt} = -\int_{\Sigma_t} H_t^2.\]
Thus,
\begin{align}
\frac{d}{dt} \Big(\phi_m(|\Sigma_t|) - |\Omega_t|\Big) &= 4\pi r^2\left(1+\frac{m}{2r}\right)^6 \frac{1}{8\pi r \left(1 + \frac{m}{2r}\right)^3\left(1 - \frac{m}{2r}\right)} \left(-\int_{\Sigma_t} H_t^2 \right) + \int_{\Sigma_t} H_t \nonumber\\
&=\sqrt{\frac{A}{16\pi}} \left(\frac{1+ \frac{m}{2r}}{1-\frac{m}{2r}}\right)\left(-\int_{\Sigma_t} H_t^2 \right) + \int_{\Sigma_t} H_t \nonumber\\
& \leq \left(A \int_{\Sigma_t} H_t^2 \right)^{1/2} \left(1- \left(\frac{1+ \frac{m}{2r}}{1-\frac{m}{2r}}\right)\left(\frac{1}{16\pi}\int_{\Sigma_t} H_t^2 \right)^{1/2}\right), \label{plugin}
\end{align}
where we used the Cauchy--Schwarz inequality for the last step. From the definition of Hawking mass \eqref{hawking} and the  assumption $m_H(\Sigma_t) \leq m$, we have
\begin{equation}\label{hawking-bound}
 \frac{1}{16\pi} \int_\Sigma H^2 \ge 1-\sqrt{\frac{16\pi m^2}{A}}. 
 \end{equation}
Since $A > 16\pi m^2$, we can take the square root of both sides and substitute this into \eqref{plugin} to obtain
\begin{align*}
\frac{d}{dt} \Big(\phi_m(|\Sigma_t| - |\Omega_t|\Big) &\leq \left(A \int_{\Sigma_t} H_t^2 dA_t\right)^{1/2} \left(1- \left(\frac{1+ \frac{m}{2r}}{1-\frac{m}{2r}}\right)\sqrt{1-\sqrt{\frac{16\pi m^2}{A}}}\right).
\end{align*}
This last expression vanishes identically. This can be checked explicitly or by observing that equality holds in \eqref{plugin} and \eqref{hawking-bound} for the coordinate spheres evolving by MCF in the Schwarzschild manifold of mass $m$, in which case $\phi_m(|\Sigma_t|) - |\Omega_t|$ is constant.
\end{proof}

We now sketch Huisken's approach to proving $\miso(M,g)\le \madm(M,g)$ (assuming nonnegative scalar curvature) using Proposition \ref{thm:huisken-monotonicity} (with some parts explained in greater detail later in the paper). This type of argument will be used in the proof of our main theorem.

Given $\epsilon>0$, we want to show that for any sufficiently large region $\Omega$, we have $\miso(\Omega) < \madm(M,g) +\epsilon$. Without loss of generality, we can assume that $\Omega$ is outward-minimizing (since replacing a set with its minimizing hull only increases the quasilocal isoperimetric mass). Moreover, we may assume that $\Sigma:=\partial \Omega$ is connected (see Lemma \ref{lemma:isop_mass} below) and lies in the asymptotically flat region of $M$. We consider a MCF $\Sigma_t=\partial\Omega_t$ with initial condition $\Sigma$. In particular, $\Sigma_t$ remains outward-minimizing. (This is explained in Section \ref{sec:weak_flow}.) 
Notice that if we set $m:=\madm(M,g)$, then by Theorem  \ref{thm:haw-adm}, we have $\mh(\Sigma_t)\le m$ for all $t$. 
\emph{Suppose} that we can flow $\Sigma_t$ until some time $t^*$ when it encloses a volume smaller than some constant $\nu>0$ independent of the initial region $\Omega$. If that were the case, then Huisken's relative volume monotonicity (Proposition \ref{thm:huisken-monotonicity}) tells us that
\[ \phi_m(|\Sigma|) - |\Omega| \ge \phi_m(|\Sigma_{t^*}|) - |\Omega_{t^*}| \ge -\nu.\] 
Then 
\begin{align*}
\miso(\Omega) &= \frac{2}{|\Sigma|} \left( |\Omega| - \frac{1}{6\sqrt{\pi}}|\Sigma|^{3/2}\right)\\
&\le \frac{2}{|\Sigma|} \left( \phi_m(|\Sigma|) +\nu - \frac{1}{6\sqrt{\pi}}|\Sigma|^{3/2}\right),
\end{align*}
which we know converges to $m$ as $|\Sigma|\to\infty$, by Lemma \ref{lemma_lim_phi_m}. Huisken suggests that this argument can be made rigorous by working solely in the asymptotically flat region where the metric is close to Euclidean, using curvature estimates for MCF \cite{Huisken:private}. However, in this paper we will take a different approach that uses weak MCF. One reason why we choose a different approach is that the $C^0$ convergence in Theorem \ref{thm:main} means that we cannot rely on curvature estimates. See Section \ref{outline} for details.

We close this section with a useful restatement of the definition of $\miso(M,g)$.

\begin{lemma}
\label{lemma:isop_mass}
Let $(M,g)$ be a $C^0$ asymptotically flat 3-manifold. The isoperimetric mass 
\[\miso(M,g) = \sup_{\{\Omega_i\}_{i=1}^\infty} \left(\limsup_{i \to \infty} \miso(\Omega_i,g)\right),\]
can be computed by taking the supremum over exhaustions by allowable regions $\Omega_i$ as in Definition \ref{def:isop_mass} with the additional property that each $\partial\Omega_i$ is smooth and connected.

Moreover, given any $\delta>0$, we can further restrict to exhaustions such that $\Omega_i$ has isoperimetric ratio at most $6\sqrt{\pi}+\delta$.
\end{lemma}

\begin{proof}
First, note that any exhausting sequence $\Omega_i$ always becomes allowable for large enough $i$.

To see that we can assume smooth boundaries, we just need to use the well-known fact that any bounded open set $\Omega$ of finite perimeter can be approximated by a set with smooth boundary whose volume and perimeter are arbitrarily close to that of $\Omega$.\footnote{This is proved in \cite[Theorem 3.42]{Ambrosio-et-al:2000} for the case of $\R^n$ by approximating the characteristic function of $\Omega$ in $L^1$ by smooth functions $u_i$ and considering super-level sets of $u_i$. A similar argument works in a smooth Riemannian manifold.  Finally, it must also hold for continuous metrics since they can be uniformly approximated by smooth metrics.} 
To see that we can assume connected boundaries, we can just connect the boundaries by thin tubes that avoid the fill-in region---a procedure that changes volume and perimeter by arbitrarily small amounts and can maintain the smoothness of the boundary.

For the second claim, let $\delta>0$. Consider an exhaustion $\{\Omega_i\}_{i=1}^\infty$, and note that $|\partial\Omega_i|\to\infty$ as $i\to\infty$ by the isoperimetric inequality. We will construct a competing sequence $\Omega_i'$ whose isoperimetric ratios are less than $6\sqrt{\pi}+\delta$, meanwhile 
\[ \limsup_{i\to\infty} \miso(\Omega_i') \ge  \limsup_{i\to\infty} \miso({\Omega}_i).\]

For large $i$, define $B_{r_i}$ to be the region enclosed by a coordinate sphere $\partial B_{r_i}$ whose area is exactly equal to $|\partial\Omega_i|$. By Lemma \ref{lemma_AF_asymptotics}, we know that $B_{r_i}$ has isoperimetric ratio less than $6\sqrt{\pi}+\delta$ for sufficiently large $i$. If $\Omega_i$ has smaller isoperimetric ratio than $B_{r_i}$, then there is nothing to do, and we can define $\Omega_i'=\Omega_i$. Otherwise, since $|\partial\Omega_i|= |\partial B_{r_i}|$, we see that $|\Omega_i|\le |B_{r_i}|$, and we can define
 $\Omega_i' = B_{r_i}$.  It then follows that $\miso(\Omega_i') = \miso(B_{r_i}) \ge \miso(\Omega_i)$, and the sequence $\Omega_i'$ has the desired property (after truncating a finite number of terms). 
\end{proof}

\section{Outline of proof and discussion of key difficulties}\label{outline}
\label{sec:outline}

In this section we outline our approach to the proof of Theorem \ref{thm:main} and describe the difficulties that must be overcome. Assume we have a sequence  $(M_i, g_i, p_i)$ converging to $(N, h, q)$ in the $C^0$ Cheeger--Gromov sense, satisfying all of the hypotheses of Theorem \ref{thm:main}. Let $\epsilon>0$, and choose a large bounded region $U\subset N$  such that $\miso(N,h) < \miso(U) + \epsilon$. Using the $C^0$ convergence of the metrics, for large enough $i$, we can find a bounded region $U_i\subset M_i$  such that $\miso(U)< \miso(U_i) +\epsilon$. The final step is to show that as long as $|\partial U|$ was chosen to be large enough, $\miso(U_i) < \madm(M_i, g_i)+\epsilon$. The main difficulty is that we need this to hold for $\epsilon$ independent of $i$.

Our approach is as follows: Choose $m= \displaystyle\liminf_{i\to\infty} \madm(M_i,g_i)+\epsilon$ and then pass to a subsequence so that $m\ge \madm(M_i, g_i)$ for large enough $i$. (Note that if $m=+\infty$, the claim is trivial.) We replace $U_i$ by its outward-minimizing hull $\tilde{U}_i$ so that $\miso(U_i)\le \miso(\tilde{U}_i)$. We then attempt to flow its boundary $\Sigma_0:= \partial\tilde U_i$ via mean curvature flow to obtain a family $\Sigma_t=\partial\Omega_t$ (where we have suppressed $i$ from the notation for clarity). As described in the previous section, \emph{suppose} that we can flow it until some time $t^*$ when it encloses a volume smaller than some constant $\nu>0$ independent of $i$. If that were the case, then since $m\ge \madm(M_i,g_i)\ge \mh(\Sigma_{t^*})$ (by Theorem \ref{thm:haw-adm}), Huisken's relative volume monotonicity (Proposition \ref{thm:huisken-monotonicity}) tells us that
\[ \phi_m(|\Sigma_0|) - |\tilde{U}_i| \ge \phi_m(|\Sigma_{t^*}|) - |\Omega_{t^*}| \ge -\nu.\] 
Then 
\begin{align}
\miso(\tilde{U}_i) &= \frac{2}{|\Sigma_0|} \left( |\tilde{U}_i| - \frac{1}{6\sqrt{\pi}}|\Sigma_0|^{3/2}\right)\nonumber\\
&\le \frac{2}{|\Sigma_0|} \left( \phi_m(|\Sigma_0|) +\nu - \frac{1}{6\sqrt{\pi}}|\Sigma_0|^{3/2}\right).\label{ineq_nu}
\end{align}
By Lemma \ref{lemma_lim_phi_m}, for large enough $|U|$, $|\tilde{U}_i|$ is large enough to ensure that this quantity is less than $m+\epsilon$, where $\epsilon$ is independent of $i$. Finally, putting everything together, 
\begin{align*}
\miso(N,h)&<\miso(U)+\epsilon\\
&< \miso(U_i)+2\epsilon\\
&\leq \miso(\tilde{U}_i)+2\epsilon\\
&< m+3\epsilon =  \liminf_{i\to\infty}\madm(M_i,g_i)+4\epsilon,
\end{align*}
which would give Theorem \ref{thm:main}.

Unfortunately, we cannot expect things to work so nicely, for a variety of reasons. For this proof to work, $\nu$ must be independent of $i$, and therefore it can only depend on $C^0$ properties of the metric. In particular, since singularities of mean curvature flow may form arbitrarily early, a weak flow is required. Moreover, we expect the weak mean curvature flow to become disconnected, in which case we no longer expect the Hawking mass to be bounded above by $\madm(M_i, g_i)$, as required by Proposition \ref{thm:huisken-monotonicity}, which is a more serious problem. We overcome this issue by noting that whenever the flow is smooth, we still have monotonicity on each component of the flow, and that when the flow disconnects the surface, this disruption essentially has ``good sign'' due to the convexity of $\phi_m(A)$ for large $A$. Unfortunately, this creates a new problem when boundary components become too small, especially when their areas dip below $16\pi m^2$, in which case  $\phi_m$ is not even well-defined. Because of this, we will introduce a modified weak mean curvature flow that ``freezes'' the components that have small area. Finally, we have to argue that our modified MCF eventually encloses a volume smaller than $\nu$, which must be independent of $i$, which is not obvious because of the components that have been frozen. In order to do this, we will show that the isoperimetric ratio of the flow is controlled.

In the end we will obtain the following refinement of one of the inequalities in Theorem~\ref{thm:iso-adm}, which will be a key ingredient in the proof of Theorem \ref{thm:main}. 
\begin{thm} \label{thm:iso-adm2}
Given constants $\mu_0 > 0$, $I_0 > 0$, $c_0>0$, there exists a constant $C$ depending only on $\mu_0, I_0,$ and $c_0$ with the following property. Let $(M,g)$ be a smooth asymptotically flat 3-manifold whose boundary is empty or minimal, with  nonnegative scalar curvature and no compact minimal surfaces in its interior, with $\madm(M,g) \leq \mu_0$. Let $\Omega$ be an outward-minimizing allowable region in $M$ with $C^{1,1}$ boundary $\partial \Omega$ that does not touch $\partial M$. Assume that $|\partial \Omega| \geq 36\pi \mu_0^2$, the isoperimetric ratio of $\Omega$ is at most $I_0$, and the isoperimetric constant\footnote{Note that this is $c(\Omega,g)$, not $c(M,g)$: the distinction is important in the proof of Theorem \ref{thm:main}, in which we will have control on the isoperimetric constant only on a bounded set. Here, $c(\Omega,g)$ is computed in the same way as $c(M,g)$, by neglecting the volume contribution of any fill-in regions in $\Omega$.} of $(\Omega,g)$ is at least $c_0$. Then \[ \miso(\Omega) \leq \madm(M,g) + \frac{C}{\sqrt{|\partial \Omega|}}. \]
\end{thm}

For our purposes, the significance of the above is that the constant $C$ depends only on coarse $C^0$ data of the metric.

\section{Convexity of $\phi_m$ and consequences}
\label{sec:convexity}

In this section we derive some key properties of $\phi_m$ that will be essential to establishing Huisken's relative volume monotonicity for a disconnected mean curvature flow (Theorem~\ref{thm_monotonicity_general}) and showing that the isoperimetric ratio of the flow is controlled (Corollary \ref{cor_isoperimetric}).

Note that $(m,A) \mapsto \phi_m(A)$ is continuous on $\{(m,A) \in \R^2\; |\; m \geq 0, A \geq 16\pi m^2\}$, and for $m \geq 0$ fixed, $A \mapsto \phi_m(A)$ is strictly increasing and twice-differentiable on $(16\pi m^2, \infty)$. It is intuitively clear that $\phi_m$ should be convex for large values of $A$, based on the $O(A^{3/2})$ asymptotics, but it cannot be convex for all $A$ since the derivative blows up to $+\infty$ at $16\pi m^2$. Precisely, we have the following.
\begin{lemma}
\label{lemma_concave_up} 
\begin{align*}
\phi''_m(A)< 0, &\quad\textrm{ for }\quad 16\pi m^2 < A< 36\pi m^2\\
\phi''_m(A) =  0,&\quad\textrm{ for }\quad A= 36\pi m^2  \\
\phi''_m(A)>0, &\quad\textrm{ for }\quad A> 36\pi m^2.
\end{align*}
\end{lemma}

\begin{proof}
Starting with equations \eqref{area} and \eqref{volume} describing $A$ and $\phi_m$ as functions of the coordinate radius $r$, we obtain
\begin{align*}
\frac{dA}{dr} &= 
8\pi r \left( 1+ \frac{m}{2r}\right)^3
\left( 1- \frac{m}{2r}\right)\\
\frac{d\phi_m}{dr} &=  \left( 1+ \frac{m}{2r}\right)^6 4\pi r^2\\
\frac{d\phi_m}{dA} &= \frac{\left( 1+ \frac{m}{2r}\right)^3 r}{2\left(   1- \frac{m}{2r}\right)}\\
\frac{d}{dr}\frac{d\phi_m}{dA} 
&= 
 \frac{  \left( 1+ \frac{m}{2r}\right)^2}{2 \left( 1- \frac{m}{2r}\right)^2} 
 \left[1+  \frac{-2m}{r}+\frac{m^2}{4r^2} \right]. 
\end{align*}
Since this last expression has the same sign as $\phi_m''(A)$, we can see by solving the quadratic that it is zero when $r=\left(1+\frac{\sqrt{3}}{2}\right)m$, negative for smaller $r$, and positive for larger $r$. The value of $A$ corresponding to $r= \left(1+\frac{\sqrt{3}}{2}\right)m$ is $36\pi m^2$.
\end{proof}

We now present some useful consequences of the convexity of $\phi_m$ for $A\ge 36\pi m^2$.
\begin{lemma}
\label{lemma_phi_multi_monotonicity}
Let $n$ be a positive integer, and let $a_k \geq b_k \geq 36\pi m^2$ be real numbers for $k=1,\ldots, n$. Let $\gamma \geq 0$ be a constant. Then
\[\phi_m\left(\gamma+ \sum_{k=1}^n a_k\right) - \sum_{k=1}^n \phi_m(a_k) \geq \phi_m\left(\gamma+ \sum_{k=1}^n b_k\right) - \sum_{k=1}^n \phi_m(b_k).\]
\end{lemma}
\begin{proof}
By Lemma \ref{lemma_concave_up}, $x \mapsto \phi_m(\gamma'+x) - \phi_m(x)$ is increasing on $[36\pi m^2, \infty)$ for any constant $\gamma'\geq0$. Using this repeatedly,
\begin{align*}
\phi_m\left(\gamma+ \sum_{k=1}^n a_k\right) - \sum_{k=1}^n \phi_m(a_k) &\geq \phi_m\left(\gamma+ b_1+\sum_{k=2}^n a_k\right) - \phi_m(b_1)-\sum_{k=2}^n \phi_m(a_k)\\
&\geq \phi_m\left(\gamma+ b_1+b_2+\sum_{k=3}^n a_k\right) - \phi_m(b_1)-\phi_m(b_2)-\sum_{k=3}^n \phi_m(a_k)\\
&\geq \ldots\\
&= \phi_m\left(\gamma+ \sum_{k=1}^n b_k\right) - \sum_{k=1}^n \phi_m(b_k).
\end{align*}
\end{proof}

The following is a monotonicity result that will be used for estimating the isoperimetric ratio in Corollary \ref{cor_isoperimetric}.
\begin{lemma}
\label{lemma_A_monotone}
Let $a \geq 0$ be a constant. Then for $A \geq 36\pi m^2$,
$$\frac{d}{dA} \left(\frac{A^{3/2}}{a+\phi_m(A)}\right) > 0.$$
\end{lemma}

\begin{proof}
\begin{align*}
\frac{d}{dA} \left(\frac{A^{3/2}}{a+\phi_m(A)}\right) &= \frac{3}{2}A^{1/2}\left(a+\phi_m(A)-\frac{2}{3} A\frac{d\phi_m}{dA} \right)(a+\phi_m(A))^{-2}.
\end{align*}
We complete the proof by establishing that $\phi_m(A)-\frac{2}{3} A\frac{d\phi_m}{dA}>0$ for $A \geq 36\pi m^2$. This is done in two parts. First, this function is increasing on $(16\pi m^2, \infty)$: borrowing our formulae for $\frac{d\phi_m}{dA}$ and $\frac{d}{dr} \frac{d\phi_m}{dA}$ from the proof of Lemma \ref{lemma_concave_up}, we compute:
\begin{align*}
\frac{d}{dA} \left(\phi_m(A)-\frac{2}{3} A\frac{d\phi_m}{dA}\right)&=\frac{1}{3}\frac{dr}{dA}\left( \frac{d \phi_m}{dr} -2A \frac{d}{dr} \frac{d\phi_m}{dA}\right)\\
&= \frac{1}{3}\frac{dr}{dA}\left(4\pi r^2\left(1+\frac{m}{2r}\right)^6-4\pi r^2 \frac{\left(1+\frac{m}{2r}\right)^6}{\left(1-\frac{m}{2r}\right)^2}\left(1-\frac{2m}{r}+\frac{m^2}{4r^2}\right)\right)\\
&=\frac{4}{3}\pi r^2 \frac{dr}{dA} \left(1+\frac{m}{2r}\right)^6\left( \frac{m}{r\left(1-\frac{m}{2r}\right)^2}\right)>0, \qquad \text{ for } r > m/2.
\end{align*}
Second, using Mathematica, it was verified that $\phi_m(A)-\frac{2}{3} A\frac{d\phi_m}{dA} \approx 19.6 > 0$ at the value $A=36\pi$ when $m=1$. Thus, since $\phi_m(A)-\frac{2}{3} A\frac{d\phi_m}{dA}$ is defined purely in terms of the geometry of the Schwarzschild manifolds of mass $m \in (0,\infty)$, which differ only by constant rescalings, this quantity is positive at $A=36\pi m^2$.
\end{proof}

\section{A modified weak mean curvature flow}
\label{sec:weak_flow}

As explained in Section \ref{sec:outline}, it will be necessary to use a weak version of mean curvature flow in order to flow past singularities. The level set formulation of weak mean curvature flow in $\R^n$, $n \geq 3$, was pioneered separately by L.C.\ Evans and J.\ Spruck \cite{Evans-Spruck:1991} and Y.G.\ Chen, Y.\ Giga, and S.\ Goto \cite{Chen-Giga-Goto:1991}. When the flow starts at a smooth, strictly mean-convex initial surface $\partial K_0$, with $K_0$ compact, the flow moves inward, and so the level set formulation can be described in terms of the \emph{arrival time} function $u:K_0 \longrightarrow [0,\infty]$, which assigns to each point $x\in K_0$ the unique time $u(x)$ at which the flow reaches that point (or $+\infty$ if the point is never reached).  In particular, $u\geq 0$ vanishes precisely on $\partial K_0$. In the case of smooth MCF, the level sets $u^{-1}(t)$ evolving by MCF translates into $u$ satisfying the equation   
\begin{equation}
\label{eqn:weak_MCF}
\Div\left( \frac{\nabla u}{|\nabla u|}\right) = \frac{-1}{|\nabla u|}.
\end{equation}
From Theorem 7.4 of \cite{Evans-Spruck:1991}, it is known that there exists a locally Lipschitz\footnote{It is now known that $u$ is actually twice differentiable \cite{Colding-Minicozzi:2015}, at least in $\R^n$.} function $u\geq 0$ satisfying (\ref{eqn:weak_MCF}) in a weak sense on $K_0$ (whose details are not essential to this work) with boundary condition $0$ at $\partial K_0$. Ilmanen proved that the same result in the setting of Riemannian manifolds with a lower bound on Ricci curvature  \cite[Theorem 6.4]{Ilmanen:1992}. We will be interested in running the weak MCF in the following situation.

\begin{assumption}
\label{assumption1}
$(M,g)$ is a smooth asymptotically flat 3-manifold whose boundary is empty or minimal, with nonnegative scalar curvature and no compact minimal surfaces in its interior. $m \geq 0$ is the ADM mass of $(M,g)$. The initial region $K_0$ (in $M^{\mathrm{fill}}$) is the closure of an outward-minimizing allowable region $\Omega_0$ whose boundary $\partial \Omega_0$ is smooth, strictly mean-convex, and disjoint from $\partial M$.
\end{assumption}

By Ilmanen's result  \cite[Theorem 6.4]{Ilmanen:1992}, under Assumption \ref{assumption1}, we can use $K_0\subset M^{\mathrm{fill}}$  as our initial data for a mean-convex level set mean curvature flow.  Note that since the initial region $K_0$ is compact with smooth boundary, the level set flow agrees with MCF for at least a short time. Although $u$ is defined on the filled-in manifold $K_0 \cap M^{\mathrm{fill}}$, we will later see that we may restrict $u$ to $K_0 \cap M$.

\subsection{Regularity properties of the level set flow}
We will require a number of regularity properties of mean-convex level set flow. 
Following \cite{White:2000}, let 
\[ \mathcal{K} =\{(x,t) \in K_0 \times [0,\infty) \; | \; u(x) \geq t\}, \] 
be the subset of ``spacetime'' contained by the flow. A point $(x,t) \in \partial \mathcal{K}$ is a \emph{regular point} if $\mathcal{K}$ is locally a smooth manifold-with-boundary near $(x,t)$ and the tangent plane to $\partial \mathcal{K}$ is not horizontal at $(x,t)$. Let $\sing(\mathcal{K})$ consist of the points of $\partial \mathcal{K}$ that are not regular (which must have $t>0$). 
By White's regularity theorem for mean convex level set flow \cite[Theorem 1.1]{White:2000}, the parabolic Hausdorff dimension of $\sing(\mathcal{K})$ is at most one. We define the  \emph{spatial singular set} $\mathcal{S}\subset K_0$ to be the projection of $\sing(\mathcal{K})$ onto the spatial factor.

\begin{lemma}\label{lemma:sing_set} Under Assumption \ref{assumption1}, for the solution $u$ to level set flow on $K_0$, 
\begin{enumerate}
\item the singular set $\mathcal{S}$ has zero Lebesgue measure in $K_0$, and
\item each level set $u^{-1}(t)$ has zero Lebesgue measure in $K_0$ (\textit{i.e.}, the flow is ``non-fattening.'')
\end{enumerate}
\end{lemma}
\begin{proof}
Since $\sing(\mathcal{K})$ has parabolic Hausdorff dimension at most one, its projection $\mathcal{S}$ has Hausdorff dimension at most one, and thus its 3-dimensional Hausdorff measure is zero, proving the first claim.  (A version of this claim was proved directly in \cite[Lemma 2.3]{Metzger-Schulze:2008} for the case $M=\R^n$.)

Now we address the second claim, which Ilmanen dubbed ``non-fattening'' \cite[Section 11]{Ilmanen:1994}. For each $t$, we know that $u^{-1}(t)\smallsetminus \mathcal{S}$ is a smooth $2$-surface, being the projection of the transverse intersection of the regular set of $\mathcal{K}$ with the $K_0 \times \{t\}$ slice. In particular, $u^{-1}(t) \smallsetminus \mathcal{S}$ has Lebesgue measure zero in $M$. Claim (2) now follows from Claim (1). 
\end{proof}

Define the compact sets $K_t = \{x \in K_0 \;|\; u(x) \geq t\}$ and the open sets $\Omega_t =\{x \in K_0 \;|\; u(x) > t\}$. By the second part of the previous lemma,
$\partial^* K_t = \partial^* \Omega_t$ for all $t \geq 0$, and we define $\Sigma_t$ to be this common reduced boundary. One may think of $\Sigma_t$ as evolving by ``weak MCF.'' 

Define the set of \emph{singular times} to be the projection of $\sing(\mathcal{K})$ to the time factor. All other times are \emph{regular times}, and at a regular time $t$, $\partial K_t$ is a smooth 2-manifold or empty. Note that at any regular time $t$, we have $\Sigma_t = u^{-1}(t) = \partial K_t =\partial\Omega_t,$ and that on any interval of regular times, the flow agrees with the classical smooth mean curvature flow. 

We have the following additional regularity properties for mean-convex mean curvature flow under Assumption \ref{assumption1}, established by White:
\begin{enumerate}
\item The set of singular times has Lebesgue measure zero in $\R$ \cite[Corollary following Theorem 1.1]{White:2000}.
\item  $\Omega_t$ is outward-minimizing (see below). 
\item $K_{t_2} \subset \interior(K_{t_1})$ if $t_2 > t_1$ \cite[Theorem 3.1]{White:2000}.
\item If $K \subset K'$ are compact sets, then the level set flow of $K$ for time $t$ is contained in the level set flow of $K'$ for time $t$ (see \cite{White:2000}, section 2.1).
\item The perimeter of $\Omega_t$ is non-increasing (follows from (2) above).
\item Long-term behavior: Either the flow vanishes in finite time, or $\Sigma_t$ converges as $t\to\infty$ to a union of minimal surfaces (and the convergence is smooth on the set of regular times) \cite[Theorem 11.1]{White:2000}. 
Under Assumption \ref{assumption1}, we know that there are no interior compact minimal surfaces in $M$.  Furthermore, by (4) and the fact that the mean curvature flow of the minimal surface $\partial M$ is stationary, $\Sigma_t$ may not cross $\partial M$ into the fill-in region. Therefore it must be the case that  $\bigcap_{t>0} K_t$ is the closure of the fill-in regions that are contained in $K_0$. In particular, the regions $\Omega_t$ are allowable, and $u$ is finite on $K_0$ except on the closure of the fill-in regions contained in $K_0$. Thus, we may regard $K_0$ as a subset of $M$, and restrict $u$ accordingly, and say that $u$ is finite on $K_0$ except at $\partial M \cap K_0$.  
\end{enumerate}

To justify (2), we have $\Omega_t \subset \tilde \Omega_t \subset \Omega_0$ by definition of minimizing hull, since $\Omega_0$ is outward-minimizing. By \cite[Theorem 3.5]{White:2000}, $\partial \tilde \Omega_t \smallsetminus \partial \Omega_t \subset K_t$, which implies $\tilde \Omega_t \subset K_t$. Since $u^{-1}(t)$ has zero Lebesgue measure by the second part of Lemma \ref{lemma:sing_set}, we conclude that $\tilde \Omega_t= \Omega_t$.

\subsection{Further properties of level set flow}
The main goal of this section is to develop a modified version of the level set flow as explained in Section \ref{sec:outline}. Before doing so, there are additional properties of level set flow we need to establish. Henceforth we use ``$P$'' to denote the perimeter (Hausdorff 2-measure of the reduced boundary).

\begin{prop}
\label{prop_u_properties}
Under the hypotheses of Assumption \ref{assumption1}, the level set flow satisfies the following properties:
\begin{enumerate}
\item[(i)] If $O \subset K_0$ is an open subset of $M$, then the infimum of $u$ on $O$ is not attained. 
\item[(ii)] $\ol{\Omega_{t_2}} \subset \Omega_{t_1}$ if $t_2 > t_1$.
\item[(iii)] $t \mapsto |\Omega_t|$ is absolutely continuous and $t \mapsto \Omega_t$ is continuous in the sense that 
$\lim_{t \to t_0} |\Omega_{t} \triangle \Omega_{t_0}| = 0$.
\item[(iv)] Let $\{U_i\}$ denote a finite or countably infinite collection of disjoint open sets such that each $U_i$ is a component of $\Omega_{t_i}$ for some $t_i \geq 0$ (where the $t_i$ can be repeated, but the $U_i$ must be distinct).
Then: 
$$P\left(\bigcup_i U_i\right) = \sum_i P( U_i).$$
\item[(v)] Suppose $t_2 > t_1$ and that $\Omega_{t_2}'$ and $\Omega_{t_1}'$ are (finite or countable) unions of components of $\Omega_{t_2}$ and $\Omega_{t_1}$, respectively. If $\Omega_{t_2}' \subset \Omega_{t_1}'$, then $P( \Omega_{t_2}') \leq P( \Omega_{t_1}')$.

\end{enumerate}
\end{prop}
\begin{proof}
$(i)$. Let $a= \inf_{x \in O} u(x)$, where $O \subset K_0$ is open in $M$. Suppose $p \in O$ attains the infimum, \textit{i.e.}, $u(p)=a$. Note that $p \notin \partial M$ since $u$ is infinite on $\partial M \cap K_0$. Let $B$ be a small closed ball around $p$, contained in $M \smallsetminus \partial M$, with smooth and strictly mean convex boundary. In particular, the mean curvature flow (and hence the level set flow) beginning at $\partial B$ is smooth for a short time. Moreover, $B \subset O \subset K_a$. However, $p$ belongs to the level set flow of $B$ for all sufficiently small times, but $p$ does not belong to $K_{t}$ for any $t > a$. This violates property (4) above, namely that compact inclusions are preserved under the level set flow for a time $t$.

$(ii)$. The inclusion $\ol \Omega_{t_2} \subset K_{t_2}$ is immediate, since $u$ is continuous. Property (3) above implies $K_{t_2} \subset \interior(K_{t_1})$.  Finally, $\interior(K_{t_1}) \subset \Omega_{t_1}$ by $(i)$, since the infimum of $u$ on $\interior(K_{t_1})$ is $\geq t_1$ and not attained.

$(iii)$. Let $\mathcal{H}^k$ denote Hausdorff $k$-measure with respect to $g$.  Note that by definition of $\sing(\mathcal{K})$, we know that $\nabla u$ exists and is non-zero away from $\mathcal{S}$. So for $t \geq 0$,
\begin{align*}
|\Omega_0|-|\Omega_t| = |\Omega_0 \smallsetminus \Omega_t| 
&= \int_{u^{-1}((0,t])} d\mathcal{H}^3 \\
&= \int_{u^{-1}((0,t])\smallsetminus\mathcal{S} } d\mathcal{H}^3 &&(\text{by  Lemma \ref{lemma:sing_set}})\\
&=  \int_{u^{-1}((0,t])\smallsetminus\mathcal{S}  } |\nabla u|^{-1} \cdot |\nabla u|\, d\mathcal{H}^3 \\
&=\int_0^t   \left[  \int_{u^{-1}(s) \smallsetminus\mathcal{S}  } |\nabla u|^{-1} d\mathcal{H}^2 \right]\, ds \\
&=\int_0^t F(s)\, ds
\end{align*}
where $F(s)$ is defined to be the integrand in brackets. Above, we used the co-area formula, which is valid since $u$ is locally Lipschitz and the integrand is nonnegative. Since $|\Omega_0|-|\Omega_t|$ is finite, $F$ is $L^1$ on any compact interval. Consequently, $|\Omega_0|-|\Omega_t|$
can be expressed as the integral of a locally integrable function of $t$, and both claims readily follow.

$(iv)$. Suppose $U_i = \Omega_{t_i}'$ is a connected component of $\Omega_{t_i}$, and let $U$ be the union of $U_i$ over $i$. For $\epsilon>0$ small, let
$$U_\epsilon = \bigcup_i \left(\Omega_{t_i}' \cap \Omega_{t_i + \epsilon}\right).$$
By uniqueness, the right-hand side is result of running the level set flow of $U$ for time $\epsilon$. By monotonicity of the perimeter (item (5) above) for this flow,
$$P(U) \geq P(U_\epsilon).$$

However, for any $\epsilon>0$, the closures of $\Omega_{t_i}' \cap \Omega_{t_i + \epsilon}$ are pairwise disjoint by $(ii)$, so that
$$P(U_\epsilon) = \sum_i P\left(\Omega_{t_i}' \cap \Omega_{t_i + \epsilon}\right).$$
Thus,
$$\liminf_{\epsilon \to 0} P(U_\epsilon) \geq \sum_i \liminf_{\epsilon \to 0} P\left(\Omega_{t_i}' \cap \Omega_{t_i + \epsilon}\right) \geq \sum_i P(\Omega_{t_i}'),$$
having used $(iii)$ and the lower semicontinuity of perimeter. Putting this all together,
$$P(U) \geq \sum_i P(U_i),$$
and the reverse inequality is immediate from the definition of perimeter.

$(v)$. Consider the level set flow with initial condition $\Omega_{t_1}'$. By uniqueness, for $t>t_1$, the flowed sets are precisely $\Omega_{t_1}' \cap \Omega_t$. By monotonicity of the perimeter under the level set flow,
$$P(\Omega_{t_1}') \geq P( \Omega_{t_1}' \cap \Omega_{t_2}).$$
The right-hand side is the perimeter of the union of \emph{all} components of $\Omega_{t_2}$ that are contained in $\Omega_{t_1}'$ (of which there are countably many, since the sum of their volumes is finite). By $(iv)$, this equals the sum of the perimeters of all such components, which is at least the sum of the perimeters of the components of $\Omega_{t_2}'$, which again by $(iv)$ is $P(\Omega_{t_2}')$.
\end{proof}

\subsection{Construction of the modified flow}
At this point we are ready to define the modified arrival time function. Recall that under Assumption \ref{assumption1}, $m \geq 0$ is the ADM mass of $(M,g)$. Intuitively, the modified flow ``freezes'' a component once its perimeter reaches $36\pi m^2$ or smaller, including situations where such a component instantly appears after a singular time. 

\begin{definition}
\label{def_u_hat}
Define $U_\infty \subset K_0$ to be the set of all $x$ such that there exists $t_0 \geq 0$ such that $x$ belongs to 
a component of $\Omega_{t_0}$ whose perimeter is $<36\pi m^2$. Then for each $x \in K_0 \subset M$, the \emph{modified arrival time} function $\hat u$ is defined to be
\[ 
\hat{u}(x):= \left\{
\begin{array}{ll}
+\infty & \text{for }x\in U_\infty \\
u(x) & \text{for }x\notin U_\infty. 
\end{array}
\right.\]
\end{definition}

Note that if $|\partial K_0| < 36\pi m^2$ (or, more generally, if all  components of $K_0$ have perimeter $<36\pi m^2$), then $\hat u$ is trivially equal to $+\infty$ on $K_0$. Also, if $m=0$, then $\hat u= u$.

It is clear that $U_\infty$ is open in $M$ because it is a union of components of open sets. We proceed to show some basic properties of the set~$U_\infty$.
\begin{prop} 
\label{prop_U_infty}
Let $U_\infty$ be as in Definition \ref{def_u_hat}.
\begin{enumerate}
\item[(A)] Each component $U'$ of $U_\infty$ equals a component of $\Omega_{t_0}$, where $t_0 = \inf_{x \in U'} u(x)$ is called the \emph{freeze time} of $U'$. 
\item[(B)] Every component of $U_\infty$ has  perimeter $\leq 36\pi m^2$.
\item[(C)] The set of all  freeze times is countable.
\item[(D)] $P(U_\infty)$ equals the sum of the perimeters of its components and is finite.
\end{enumerate}
\end{prop}

An illustration is given in Figure \ref{fig_flow}.
\begin{figure}[ht]
\begin{center}
\includegraphics[scale=0.60]{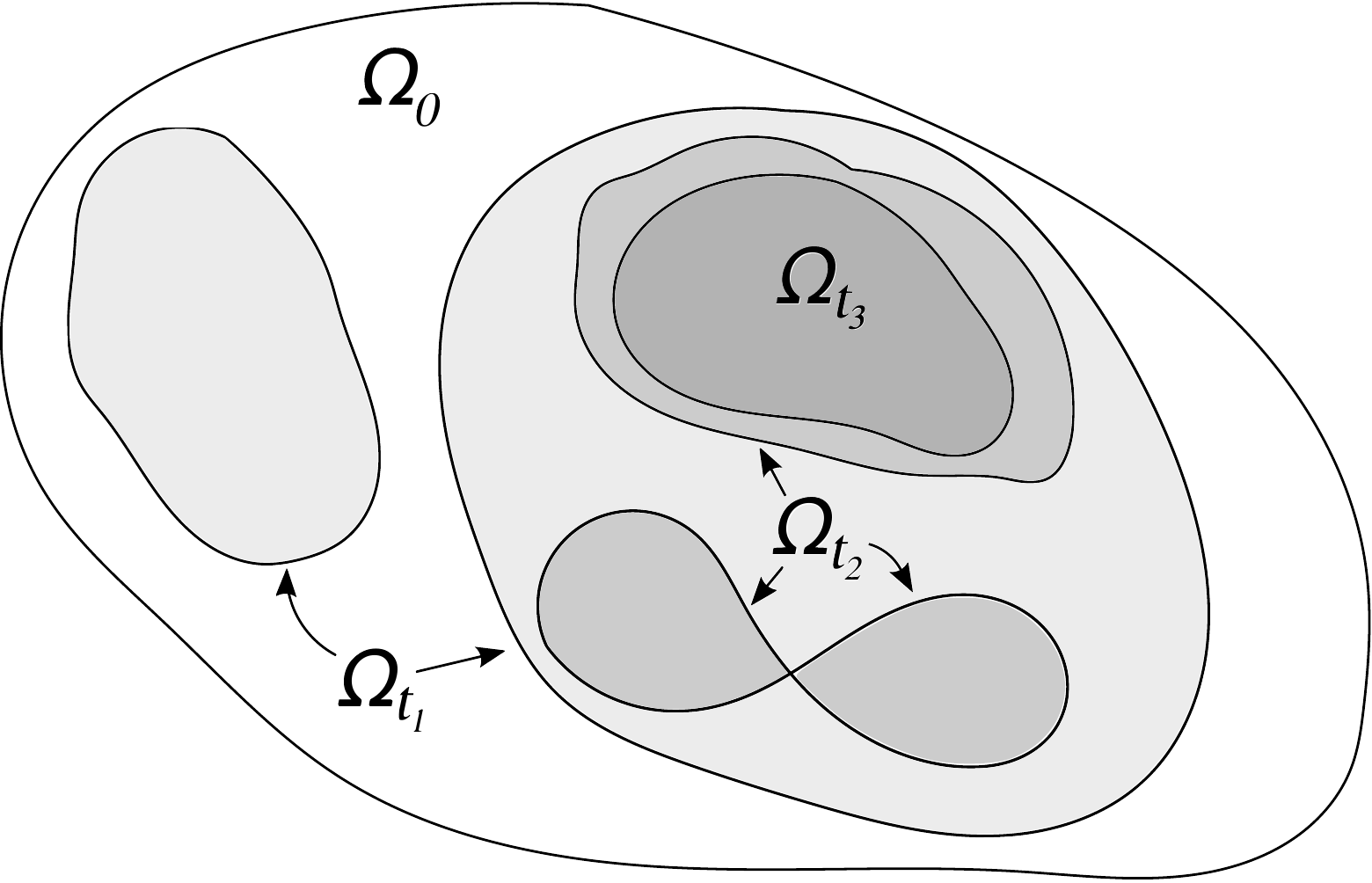}
\caption{\small A schematic of a possible state of the flow at times $0< t_1 < t_2 < t_3$. Time in the flow is correlated with the level of shading. In this example, suppose that the left component of $\Omega_{t_1}$  has perimeter $36\pi m^2$,  the ``dumbbell'' has perimeter $>36\pi m^2$ but the individual halves have perimeter $\le 36\pi m^2$, and that $\Omega_{t_3}$ has 
has perimeter $36\pi m^2$. In this scenario, $t_1$, $t_2$, and $t_3$ are freeze times, and $U_\infty$ is the union of the left component of $\Omega_{t_1}$, the two dumbbell components of $\Omega_{t_2}$, and $\Omega_{t_3}$. \label{fig_flow}}
\end{center}
\end{figure}

\begin{proof}
We begin with some notation that will be used a number of times in the proof. Fix any $t_0 \geq 0$ and a component $\Omega_{t_0}'$ of $\Omega_{t_0}$. For $\epsilon>0$, let $\Omega_{t_0+\epsilon}' = \Omega_{t_0}' \cap \Omega_{t_0+\epsilon}$, which is just the level set flow of $\Omega_{t_0}'$ for time $\epsilon$. Note that
\begin{equation}
\label{eqn:omega_epsilon}
\Omega_{t_0}' = \bigcup_{\epsilon > 0} \Omega_{t_0 + \epsilon}'.
\end{equation}

Claim: For all $\epsilon>0$ sufficiently small, $\Omega_{t_0+\epsilon}'$ is connected. 

Proof of the claim:  By property $(ii)$ of Proposition \ref{prop_u_properties}, (\ref{eqn:omega_epsilon}) is an increasing union as $\epsilon \searrow 0$. Now,
the claim follows from general topology and the fact that $\Omega_{t_0}'$ is connected.

$(A)$. Fix a component $U'$ of $U_\infty$, and let $t_0 = \displaystyle\inf_{x \in U'} u(x)$ be its freeze time. 
So if $x \in U'$, then $u(x) > t_0$ by $(i)$ of Proposition \ref{prop_u_properties} (since $U'$ is open), and thus $x \in \Omega_{t_0}$. Since $U'$ is connected, $U' \subset \Omega_{t_0}$ is contained in a unique component $\Omega_{t_0}'$ of $\Omega_{t_0}$, that is, $U' \subset \Omega_{t_0}'$. We proceed to show that equality holds.

First, if $U' \subset \Omega_{t_0 + \epsilon}'$ for any $\epsilon > 0$, then $\displaystyle\inf_{x \in U'} u(x) \geq t_0+\epsilon> t_0$, a contradiction. Thus, we may choose $\epsilon_0>0$ sufficiently small so that for all $\epsilon \in (0,\epsilon_0]$, $\Omega'_{t_0+\epsilon}$ is connected (by the claim) and does not contain $U'$. Please see Figure \ref{fig_omega_t_0} for an illustration.

\begin{figure}[ht]
\begin{center}
\includegraphics[scale=0.65]{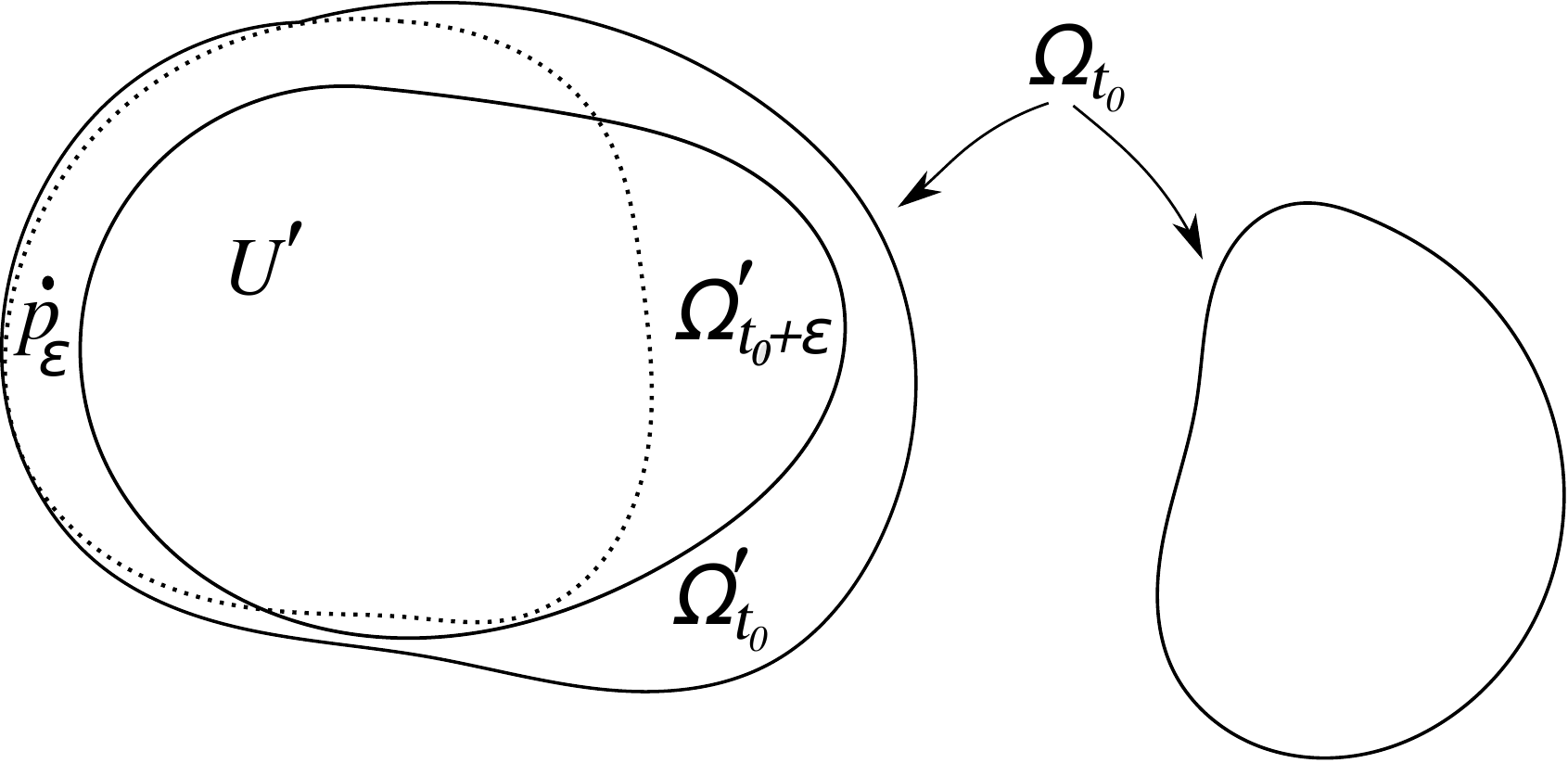}
\caption{\small An illustration of the sets used in the proof of (A) of Proposition \ref{prop_U_infty}. The region with the dotted outline is $U'$. \label{fig_omega_t_0}}
\end{center}
\end{figure}

For each $\epsilon \in (0, \epsilon_0]$, there exists a point $p_\epsilon \in U' \smallsetminus \Omega_{t_0+\epsilon}'$. 
Since $p_\epsilon \in U' \subset U_\infty$, there exists some $t(\epsilon)\geq 0$ and a component $\Omega_{t(\epsilon)}''$ of $\Omega_{t(\epsilon)}$ that contains $p_\epsilon$ and has perimeter $< 36\pi m^2$ (by definition of $U_\infty$).  Note that $t(\epsilon) < u(p_\epsilon) \leq t_0 + \epsilon$.

If $t(\epsilon) < t_0$ for any of the $\epsilon$, then $\Omega_{t(\epsilon)}''$  contains $\Omega_{t_0}'$, since the latter is connected and  both include $p_\epsilon$. By $(v)$ of Proposition \ref{prop_u_properties}, $P( \Omega_{t_0}') \leq P(\Omega_{t(\epsilon)}'') < 36\pi m^2$, which implies $\Omega_{t_0}' \subset U'$, and we are done. 

Thus, we may assume $t(\epsilon) \in [t_0, t_0 + \epsilon_0]$ for every value of $\epsilon \in (0,\epsilon_0]$. Since $p_\epsilon \in \Omega_{t_0}' \cap \Omega_{t(\epsilon)}=\Omega_{t(\epsilon)}'$, we must have $\Omega_{t(\epsilon)}' = \Omega_{t(\epsilon)}''$ by connectedness. It is clear that
$$\Omega_{t_0 + \epsilon}' \subset \Omega_{t(\epsilon)}',$$
since $t_0 + \epsilon \geq t(\epsilon)$.
Also, since $\Omega_{t(\epsilon)}'$ has perimeter $< 36\pi m^2$,
$$\Omega_{t(\epsilon)}' \subset U'.$$
Thus, $\Omega_{t_0 + \epsilon}' \subset U'$ for every $\epsilon \in (0,\epsilon_0]$. By (\ref{eqn:omega_epsilon}), $\Omega_{t_0}'$, is a subset of $U'$.

$(B)$. Let $U'$ be a component of $U_\infty$, so  $U'= \Omega_{t_0}'$ for some $t_0\geq 0$ by $(A)$. The argument in (A) shows that either $P(U') < 36\pi m^2$ or else $U'$ is the limit of sets $\Omega_{t(\epsilon)}'$ (as $\epsilon \searrow 0$) of perimeter $<36\pi m^2$. By lower semicontinuity of perimeter, $P(U') \leq 36\pi m^2$.

$(C)$. From the definition, the cardinality of the set of freeze times is at most the cardinality of the set of components of $U_\infty$, which is at most countably infinite (since the sum of the volumes of the components of $U_\infty$ is finite).

$(D)$. That $P(U_\infty)$ equals the sum of the perimeters of its components follows from $(A),(C)$, and  $(iv)$ of Proposition \ref{prop_u_properties}.
Now we only need to show that $P(U_\infty)$ is finite.

We first prove a weaker statement, restricting to a single freeze time. Let $t_1$ be a freeze time, and let $\Omega_{t_1}^{(1)}, \Omega_{t_1}^{(2)}, \ldots, \Omega_{t_1}^{(n_1)}$ denote a finite collection of distinct components of $U_\infty$ that have freeze time $t_1$. Then, using $(v)$ of Proposition \ref{prop_u_properties} and the first part of $(D)$,

\begin{align*}
P( \Omega_0) & \geq P \left(\Omega_{t_1}^{(1)} \cup \Omega_{t_1}^{(2)} \cup \ldots \cup \Omega_{t_1}^{(n_1)}\right)\\
&=P( \Omega_{t_1}^{(1)}) + P( \Omega_{t_1}^{(2)}) + \ldots + P( \Omega_{t_1}^{(n_1)}).
\end{align*}
Since this holds for all finite collections of components with freeze time $t_1$, the proof of the weaker statement is complete: the sum of the perimeters of all components of $U_\infty$ corresponding to a single freeze time is finite.

In general, let $t_1 < t_2 < \ldots < t_N$ denote any finite collection of freeze times. For each $k=1,2,\ldots,N$, let $\Omega_{t_k}^{(1)}, \Omega_{t_k}^{(2)}, \ldots, \Omega_{t_k}^{(n_k)}$ denote finite collections of distinct components of $U_\infty$ that have freeze time $t_k$. This general case follows from induction, making repeated use of Proposition \ref{prop_u_properties}(v).
\end{proof}

There is one further topological result that we need.
\begin{lemma}
\label{lemma_connected_boundary}
Given Assumption \ref{assumption1}, let $t_0\geq 0$ be a regular time of the level set flow. Then every component of $\Omega_{t_0}$ has connected boundary.
\end{lemma}
\begin{proof}
Suppose $t_0$ is a regular time. Thus $\partial \Omega_{t_0}$ is smooth, and we take it be non-empty (or else the result is trivial). 
Let $\Omega_{t_0}'$ be a component of $\Omega_{t_0}$ such that $\partial \Omega_{t_0}'$ is disconnected. 
Let $S$ be any component of $\partial\Omega_{t_0}'$. 

Note that $\partial \Omega_{t_0}$ does not touch $\partial M$, because $u$ is finite on the former and infinite on the latter. Thus,
if we define $X$ to be the closure of $\Omega_{t_0}$ minus the fill-in region, then $X$ is a smooth manifold-with-boundary. It is a standard result in geometric measure theory that there exists a surface $\tilde S$ in $X$ that minimizes area in the homology class of $S$ in $X$. Note that $\tilde S$ is non-empty, or else it would not be homologous to $S$. Also note that $\tilde S$ may not intersect the outer boundary $\partial \Omega_{t_0}'$, by the maximum principle, because the latter is mean-convex. If any component of $\tilde S$ does not intersect $\partial M$, then that component would be an interior minimal surface, a contradiction. Thus every component of $\tilde S$ touches $\partial M$; since $\partial M$ is minimal, we have $\tilde S \subset \partial M$. This contradicts the fact that $\tilde S$ is homologous to $S \subsetneq \partial \Omega_{t_0}'$.
\end{proof}

Now we are ready to define the modified level set flow.
\begin{definition}
Given the modified level set function $\hat u$ from Definition \ref{def_u_hat}, define for $t\geq 0$:
\begin{align*}
\hat \Omega_t &= \{x \in K_0 \; | \; \hat u(x) > t\} = \Omega_t \cup U_\infty,
\end{align*}
which is open.
\end{definition}

\section{Properties of the modified flow}
\label{sec:mod_flow_prop}

In this section we continue with Assumption \ref{assumption1}. Most of the regularity properties of the level set flow carry over immediately to the modified flow, because $u = \hat u$ on $K_0 \smallsetminus U_\infty$. For instance, $\hat \Omega_t$ has non-increasing perimeter in $t$. Also:

\begin{lemma}
\label{lemma:AC}
$|\hat \Omega_t|$ is absolutely continuous as a function of $t \geq 0$.
\end{lemma}
\begin{proof}
Following the proof of Proposition \ref{prop_u_properties}(iii), 
\begin{align*}
|\hat \Omega_0|-|\hat \Omega_t| = |\hat \Omega_0 \smallsetminus \hat \Omega_t| 
&= \int_{\hat u^{-1}((0,t])} d\mathcal{H}^3 \\
&= \int_{u^{-1}((0,t]) \smallsetminus U_\infty} d\mathcal{H}^3,
\end{align*}
and the rest of the proof follows analogously, applying the co-area formula to $u$ on the compact set $K_0 \smallsetminus U_\infty$.
\end{proof}

The following result addresses the long-time behavior of the modified flow.
\begin{lemma}[Flow freezes in finite time]
\label{lemma_area}
Let $(M,g)$ and $K_0$ be as in Assumption \ref{assumption1}, and let $\hat \Omega_t$ be the modified level set flow. There exists a unique, finite $T \geq 0$ such that  $\hat \Omega_t = U_\infty$ for $t\geq T$, while $\hat \Omega_t \supsetneq U_\infty$ for $t<T$.
\end{lemma}

Note that if at least one component of $K_0$ has perimeter $>36\pi m^2$, then $T>0$.
\begin{proof}
We claim that for $t$ sufficiently large, $\Omega_t$ has perimeter $< 36\pi m^2$. If $\partial M \cap K_0 = \emptyset$, then  $\Omega_t$ is empty for large $t$ (by property (6) of the level set flow). Otherwise, suppose $\partial M \cap K_0 \neq \emptyset$. The Riemannian Penrose inequality (Theorem \ref{thm:RPI}) guarantees that every component of $\partial M$ has area $\leq 16 \pi m^2 < 36\pi m^2$ (since $\partial M$ is an outermost minimal surface and $(M,g)$ has nonnegative scalar curvature).  From property (6) of the level set flow (Theorem 11.1 of \cite{White:2000}), we know that $\partial K_t$ smoothly converges to 
 $\partial M \cap K_0$ as $t\to\infty$. In particular, for sufficiently large $t$, each component of $\partial K_t$ has area $<36\pi m^2$.  By Lemma \ref{lemma_connected_boundary}, every component of $\Omega_t$ has perimeter $<36\pi m^2$, for sufficiently large $t$. 
 
By the claim, for large enough $t$, $\Omega_t\subset U_\infty$ and thus $\hat \Omega_t = U_\infty$. Finally, taking $T$ to be the infimum of all of the $t$ satisfying $\hat \Omega_t = U_\infty$, we see that $\hat\Omega_T$ is the union of those $\hat\Omega_t$, and thus $\hat \Omega_T = U_\infty$ as well. The result follows.
\end{proof}

Define the set of singular times for the modified flow to be the set of singular times for the original flow in $[0,T]$, union the set of freeze times. Note that this set of singular times has measure zero (by property (1) of the level set flow and Proposition \ref{prop_U_infty}(C)). We define the regular times for the modified flow to be the complement in $[0,T]$ of the set of singular times.
$\;$\\

\noindent \textbf{Notation.} From this point on, we drop the hat notation and use $\Omega_t$ to refer to the modified level set flow. 

\begin{prop}[Huisken's relative volume monotonicity for the modified flow]
\label{thm_monotonicity_general}
Let $(M,g)$ and $m$ be as described in Assumption \ref{assumption1}, let  $\Omega_t$ be the modified level set flow, and let $T$ be as in Lemma \ref{lemma_area}.
Then $\phi_m(|\partial^* \Omega_t|) - |\Omega_t|$ is non-increasing over $[0,T]$.
\end{prop}
Note that $\phi_m(|\partial^* \Omega_t|)$ is defined for $t \in [0,T)$: if $t < T$, the flow is not completely frozen, so that there exists a component of $\Omega_t$ with perimeter $\geq 36\pi m^2 > 16\pi m^2$. If $|\partial^* \Omega_T| < 16\pi m^2$, we define $\phi_m(|\partial^* \Omega_T|)$ to be zero (in which case $t \mapsto \phi_m(|\partial^* \Omega_t|)$ remains monotone on $[0,T]$).

\begin{proof}
We first establish the claim over an interval $[t_1,t_2]$ that contains no singular times of the modified flow. In particular, there are no freeze times in $[t_1, t_2]$, though there may be some components that are already frozen (and thus unchanging in the interval $[t_1, t_2]$), while the remaining components just evolve under smooth, classical mean curvature flow. For $t \in [t_1,t_2]$, let  $\Omega^1_t,\ldots,\Omega^\ell_t$ denote the \emph{unfrozen} components of $\Omega_t$, and let $\Omega^0$ denote the union of the frozen components. By a frozen component, we just mean a component that has perimeter $\le 36\pi m^2$, or equivalently, a component that equals one of the components of $U_\infty$. As mentioned, $\partial \Omega^1_t,\ldots,\partial \Omega^\ell_t$ are smooth and evolve by smooth MCF.
We know that there are only finitely many components, because at a regular time, $\partial \Omega_t$ must be a compact surface.

Define $V(t)=|\Omega_t|$, $V_i(t)=|\Omega^i_t|$ for each $i=1,\ldots,\ell$, and $V_0=|\Omega^0|$. We also define $A(t)=|\partial^* \Omega_t|$, $A_i(t)=|\partial^*\Omega^i_t|$ for each $i=1,\ldots,\ell$, and $A_0=|\partial^*\Omega^0|$.  Note that each $\partial \Omega^i_t$ must be connected (by Lemma \ref{lemma_connected_boundary}) and have area $>36\pi m^2$ (or else it would be frozen). Also, each $\Omega^i_t$ is outward-minimizing, since it is a component of the unmodified flow (which is outward-minimizing by property (2) of the level set flow). Thus, the Hawking mass of $\partial \Omega^i_t$ is bounded above by $m$ by Theorem \ref{thm:haw-adm}. Applying Huisken's relative volume monotonicity (Proposition \ref{thm:huisken-monotonicity}) to each unfrozen component, we have for each $i$, 
\[ \phi_m(A_i(t_1)) - V_i(t_1)\ge  \phi_m(A_i(t_2)) - V_i(t_2).\]
Therefore
\begin{align}
 \sum_{i=1}^\ell  [\phi_m(A_i(t_1))- V_i(t_1)] -V_0 &\ge
 \sum_{i=1}^\ell  [\phi_m(A_i(t_2)) - V_i(t_2) ] - V_0  \nonumber \\
  \sum_{i=1}^\ell  \phi_m(A_i(t_1))-V(t_1) &\ge \sum_{i=1}^\ell  \phi_m(A_i(t_2))   - V(t_2). \label{first-thing}
 \end{align}
On the other hand, using the fact $36\pi m^2 \le A_i(t_2) < A_i(t_1)$, we can apply Lemma \ref{lemma_phi_multi_monotonicity} to see that 
\begin{align}
 \phi_m\left( A_0 + \sum_{i=1}^\ell A_i(t_1)\right) - \sum_{i=1}^\ell  \phi_m(A_i(t_1)) &\ge  \phi_m\left( A_0 + \sum_{i=1}^\ell A_i(t_2)\right) - \sum_{i=1}^\ell  \phi_m(A_i(t_2)) \nonumber\\
  \phi_m(A(t_1)) - \sum_{i=1}^\ell  \phi_m(A_i(t_1)) &\ge  \phi_m(A(t_2)) - \sum_{i=1}^\ell  \phi_m(A_i(t_2)). \label{second-thing}
\end{align}
Adding together \eqref{first-thing} and \eqref{second-thing} yields the desired result, under the smoothness assumption.

In general, we know that $\phi_m(A(t))$ and $V(t)$ possess derivatives almost everywhere on $[0,T]$, by monotonicity. Since $V(t)$ is also absolutely continuous (by Lemma \ref{lemma:AC}), we have for $t \in [0,T]$,
\begin{align*}
V(t) - V(0) &= \int_0^t V'(t) dt\\
\phi_m(A(t)) - \phi_m(A(0)) &\leq \int_0^t \frac{d}{dt} \phi_m(A(t)) dt.
\end{align*}
Thus, 
\[ [\phi_m(A(t)) - V(t)] - [\phi_m(A(0)) - V(0)] \leq \int_0^t \frac{d}{dt} [\phi_m(A(t)) - V(t)] \,dt.\]
Now let $R \subset [0,T]$ be the set of regular times of the modified flow. For each $t_0\in R$, the unfrozen components of $\partial\Omega_t = \partial K_t$ are smooth and have perimeter $>36\pi m^2$. So for some $\epsilon>0$, the modified flow on the interval $[t_0, t_0+\epsilon)$ simply flows those components smoothly without any new freezing occurring. We can apply the smooth case to see that $\frac{d}{dt} \left(\phi_m(A(t)) - V(t)\right)\le 0$ at $t=t_0$, completing the proof, since the measure of $[0,T] \smallsetminus R$ is zero.
\end{proof}

\begin{cor}[Control of the isoperimetric ratio]
\label{cor_isoperimetric} 
Let $(M,g)$ and $m$ be as described in Assumption \ref{assumption1}, let  $\Omega_t$ be the modified level set flow, and let $T$ be as in Lemma \ref{lemma_area}. Define $I(t) := \frac{|\partial^* \Omega_t|^{3/2}}{|\Omega_t|}$, the isoperimetric ratio of $\Omega_t$.
Suppose $\phi_m(|\partial^* \Omega_0|) \leq |\Omega_0|$. Then for $t \in [0, T]$,
$I(t) \leq I(0)$.
\end{cor}
\begin{proof}With notation as in the proof of Proposition \ref{thm_monotonicity_general},
\begin{align*}
I(t) &= \frac{A(t)^{3/2}}{V(t)}\\
&\leq \frac{A(t)^{3/2}}{V(0)-\phi_m(A(0))+\phi_m(A(t))}\\
&\leq \frac{A(0)^{3/2}}{V(0)-\phi_m(A(0))+\phi_m(A(0))}\\
&=I(0),
\end{align*}
where we have used Proposition \ref{thm_monotonicity_general}, Lemma \ref{lemma_A_monotone} (with the choice of $a =  V(0)- \phi_m(A(0)) \geq 0$), and the fact that $A(t) \leq A(0)$.
\end{proof}
Although Corollary \ref{cor_isoperimetric} only applies if $\phi_m(|\partial^* \Omega_0|) \leq |\Omega_0|$, it turns out that this is the only case in which it is needed.

\section{Proof of Theorem \ref{thm:iso-adm2}}
\label{sec:proof1}

In order to prove Theorem \ref{thm:iso-adm2}, we must connect the area bound on components of the final state $\Omega_T$ of the modified level set flow (see Proposition \ref{prop_U_infty}(B) and Lemma \ref{lemma_area}) to a volume bound. The complication here is that we only have perimeter bounds for the individual components of $\Omega_T$ rather than for the total perimeter. The following lemma takes care of this complication.

\begin{lemma}
\label{lemma_isoperimetric}
Let $(M, g)$ be a Riemannian 3-manifold with positive isoperimetric constant $c=c(M,g)$. 
Fix a constant $\alpha > 0$. Suppose $\Omega \subset M$ is a bounded, open set whose perimeter is finite and equals the sum of the perimeters of its components.
Suppose that every component of $\Omega$ has perimeter at most $\alpha$.
Then
$$|\Omega| \leq  c^{-3} \alpha^{3/2} I^2,$$
where $I = \frac{|\partial^* \Omega|^{3/2} }{ |\Omega|}$ is the isoperimetric ratio of $\Omega$. 
\end{lemma}

\begin{proof}
Since $\Omega$ is an open set of finite volume, it has at most countably many components. By assumption, 
the perimeters of these components can be written as $\alpha r_1, \alpha r_2, \alpha r_3, \ldots$, where each $r_k\le 1$, and also
\[ |\partial^* \Omega| = \sum_{k=1}^\infty \alpha r_k.\]
Using the isoperimetric constant $c=c(M,g)$, we can bound the volume of each component in terms of its perimeter to obtain
\begin{align*}
|\Omega| & \le \sum_{k=1}^\infty c^{-1} (\alpha r_k)^{3/2} \\
& \le  \sum_{k=1}^\infty c^{-1} \alpha^{3/2}  r_k \\
& = c^{-1} \alpha^{1/2} |\partial^* \Omega| \\
& = c^{-1} \alpha^{1/2} I^{2/3} |\Omega|^{2/3},
\end{align*}
where we used the fact that $r_k\le 1$ and the definition of $I$. The result now follows by cubing and dividing by $|\Omega|^2$.
\end{proof}

\begin{lemma}
\label{thm_volume}
Let $(M,g)$ and $m$ be as described in Assumption \ref{assumption1}, let  $\Omega_t$ be the modified level set flow, and let $T$ be as in Lemma \ref{lemma_area}. If $\phi_m(|\partial\Omega_0|)\le |\Omega_0|$, then the final state $\Omega_T$ of the modified level set flow satisfies
\[ |\Omega_T| \leq c^{-3} (36\pi m^2)^{3/2} I^2,\] where $c$ is the isoperimetric constant of $(\Omega_0, g)$, and $I$  is the isoperimetric ratio of $\Omega_0$.
\end{lemma}

\begin{proof}
By Corollary \ref{cor_isoperimetric}, the isoperimetric ratio $I(T)$ of $\Omega_T$ is bounded above by the isoperimetric ratio $I$ of $\Omega_0$. By Lemma \ref{lemma_area} and Proposition \ref{prop_U_infty}(B),  each component of the final state $\Omega_T$ has perimeter $\leq 36\pi m^2$. By Lemma \ref{lemma_isoperimetric}, the volume $|\Omega_T|$ of the final state of the flow is therefore bounded above by $c^{-3} (36\pi m^2)^{3/2} I^2$, where $c$ is the isoperimetric constant of $(\Omega_0,g)$.  (Here, we used Proposition \ref{prop_U_infty}(D) to be sure that the hypotheses of Lemma \ref{lemma_isoperimetric} are satisfied.)
\end{proof}

Finally, we have all the  ingredients needed to prove Theorem \ref{thm:iso-adm2}.
\begin{proof}[Proof of Theorem \ref{thm:iso-adm2}]
Fix constants $\mu_0>0, c_0>0,I_0>0$. 
Let $(M,g)$ be a smooth asymptotically flat 3-manifold whose boundary is empty or minimal, with  nonnegative scalar curvature and interior compact minimal surfaces in its interior.
 Let $\Omega$ be an outward-minimizing allowable region in $M$ whose boundary does not touch $\partial M$.  At first, we assume that $\partial \Omega$ is smooth and strictly mean-convex. Assume $|\partial \Omega| \geq 36\pi \mu_0^2$, $m:=m_{ADM}(M,g) \leq \mu_0$, $c:=c(M,g) \geq c_0$, and the isoperimetric ratio $I(\Omega)$ is $ \leq I_0$.  Note that Assumption \ref{assumption1} holds for $K_0 = \overline{\Omega}$.

First, by Proposition \ref{thm_monotonicity_general},
\begin{align}
\miso(\Omega) &= \frac{2}{|\partial \Omega|}\left( |\Omega| - \frac{1}{6\sqrt{\pi}}|\partial \Omega|^{3/2}\right)\label{eqn_def_miso_omega}\\
&\leq \frac{2}{|\partial \Omega|}\left( \phi_m(|\partial \Omega|) +|\Omega_T|- \frac{1}{6\sqrt{\pi}}|\partial \Omega|^{3/2}\right),\nonumber
\end{align}
where $\Omega_T$ is the final state of the modified level set flow beginning at $\ol \Omega$.  (Note $\phi_m(|\partial \Omega|)$ is defined, because $|\partial \Omega| \geq 36\pi \mu_0^2 \geq 16\pi m^2.$) Now, if $\phi_m(|\partial \Omega|) \leq |\Omega|$, then by Lemma \ref{thm_volume},
\begin{align}
\miso(\Omega) &\leq \frac{2}{|\partial \Omega|}\left( \phi_m(|\partial \Omega|) +c^{-3} (36\pi m^2)^{3/2} I(\Omega)^2 -\frac{1}{6\sqrt{\pi}}|\partial \Omega|^{3/2}\right)\nonumber\\
&\leq \frac{2}{|\partial \Omega|}\left( \phi_m(|\partial \Omega|) +c_0^{-3} (36\pi \mu_0^2)^{3/2} I_0^2 -\frac{1}{6\sqrt{\pi}}|\partial \Omega|^{3/2}\right).\label{eqn:Omega}
\end{align}
 On the other hand, if $ |\Omega|< \phi_m(|\partial \Omega|)$, then (\ref{eqn:Omega}) follows trivially from the definition of $\miso$. By Lemma \ref{lemma_lim_phi_m}, we now have
\begin{align*}
\miso(\Omega) &\leq m + O(|\partial \Omega|^{-1/2}) +\frac{2c_0^{-3} (36\pi \mu_0^2)^{3/2} I_0^2}{|\partial \Omega|}\\
&\leq m + O(|\partial \Omega|^{-1/2}) +\frac{2c_0^{-3} (36\pi \mu_0^2) I_0^2}{|\partial \Omega|^{1/2}},
\end{align*}
since $|\partial \Omega| \geq 36\pi \mu_0^2$, where ``$O$'' depends only on $\mu_0$. The result now follows, under the assumption that $\partial \Omega$ is smooth and strictly mean-convex. 

Last, if $\partial \Omega$ is merely $C^{1,1}$ and/or not strictly mean convex, we apply a smoothing process. By \cite[Lemma 5.6]{Huisken-Ilmanen:2001} and the fact $\partial \Omega$ can be pushed inward without touching $\partial M$, $\Omega$ may be approximated from the inside by outward-minimizing allowable regions $\Omega^{(\epsilon)}$ with smooth boundary and strictly positive mean curvature, where $\partial \Omega^{(\epsilon)} \to \partial \Omega$ in $C^1$ as $\epsilon \to 0$. In particular, applying the above argument for $K_0 = \overline{\Omega^{(\epsilon)}}$ and letting $\epsilon \to 0$ suffices to establish the same bound\footnote{An alternative to the smoothing argument is to use an extension of the level flow developed by Metzger and Schulze for an initial region whose boundary is merely $C^{1}$ with nonnegative weak mean curvature in $L^2$ \cite{Metzger-Schulze:2008}. Their work assumes a Euclidean ambient space but is expected to generalize to a Riemannian manifold.}.
\end{proof}

\section{Proof of Theorem \ref{thm:main}}
\label{sec:proof2}

Assume the hypotheses of Theorem \ref{thm:main}, and let $m_i:=\madm(M_i,g_i)$.
The claim is trivial if $\liminf_{i \to \infty} m_i= +\infty$. Thus, without loss of generality, we may pass to a subsequence and assume that $\{m_i \}$ is uniformly bounded above, say by a constant $\mu>0$.  Note that each $m_i \geq 0$  by the positive mass theorem.

Let $\epsilon>0$. Choose a constant $R_0$ sufficiently large so that
\begin{align}
\pi R_0^2 & \geq 36\pi \mu^2  \label{eqn_pi_mu}\\
\frac{C}{\sqrt{\pi}R_0} &< \epsilon, \label{eqn_C_R_0}
\end{align}
where $C$ is the constant in Theorem \ref{thm:iso-adm2} corresponding to the upper bound $\mu$ for the ADM mass, the lower bound on the isoperimetric constants given by $\tfrac{1}{2}c(N,h)$, which is positive by Lemma \ref{sop-lower-bound}, and the upper bound of $8\sqrt{\pi}$ for the isoperimetric ratio. Fix a $C^0$ asymptotically flat coordinate chart for $(N,h)$.

For now, assume that $\miso(N,h)<\infty$. Let $B_r$ denote the bounded open set enclosed by the coordinate sphere $\partial B_r$ of radius $r$. 
Choose an allowable region $\Omega \subset N$ so that
\begin{enumerate}[(i)]
\item $\miso(N,h) < \miso(\Omega,h) + \epsilon$.
\item $B_{R_0}\subset \Omega $. In particular, $\Omega$ contains $\partial N$.
\item $\partial\Omega$ is smooth and connected, and the isoperimetric ratio of $\Omega$ with respect to $h$ is at most $7\sqrt{\pi}$. (Here, we used Lemma \ref{lemma:isop_mass}.) In particular, $\Omega$ is connected.
\item The ratio of areas measured by $h$ and $\delta$ on $N \smallsetminus \Omega$ is bounded above by $2$ (which is possible by $C^0$ asymptotic flatness).
\end{enumerate}
Choose $R_1>0$ large enough so that $\Omega\subset B_{R_1}$ and 
\begin{equation}\label{size-R}
|\partial\Omega|_h + 8\pi\mu^2 < \frac{\pi}{32}R_1^2.
\end{equation}
 Now choose $R_2>0$ such that the ball $B_h(q,R_2)$ contains $B_{3R_1}$, and take the embeddings $\Phi_i: U \supset B_h(q,R_2) \longrightarrow M_i$ in accordance with the definition of pointed $C^0$ Cheeger--Gromov convergence (Definition \ref{def:CG}), for $i \geq$ some $i_0$. See Figure \ref{fig_lsc} for a depiction of the setup.

\begin{figure}[ht]
\begin{center}
\includegraphics[scale=0.50]{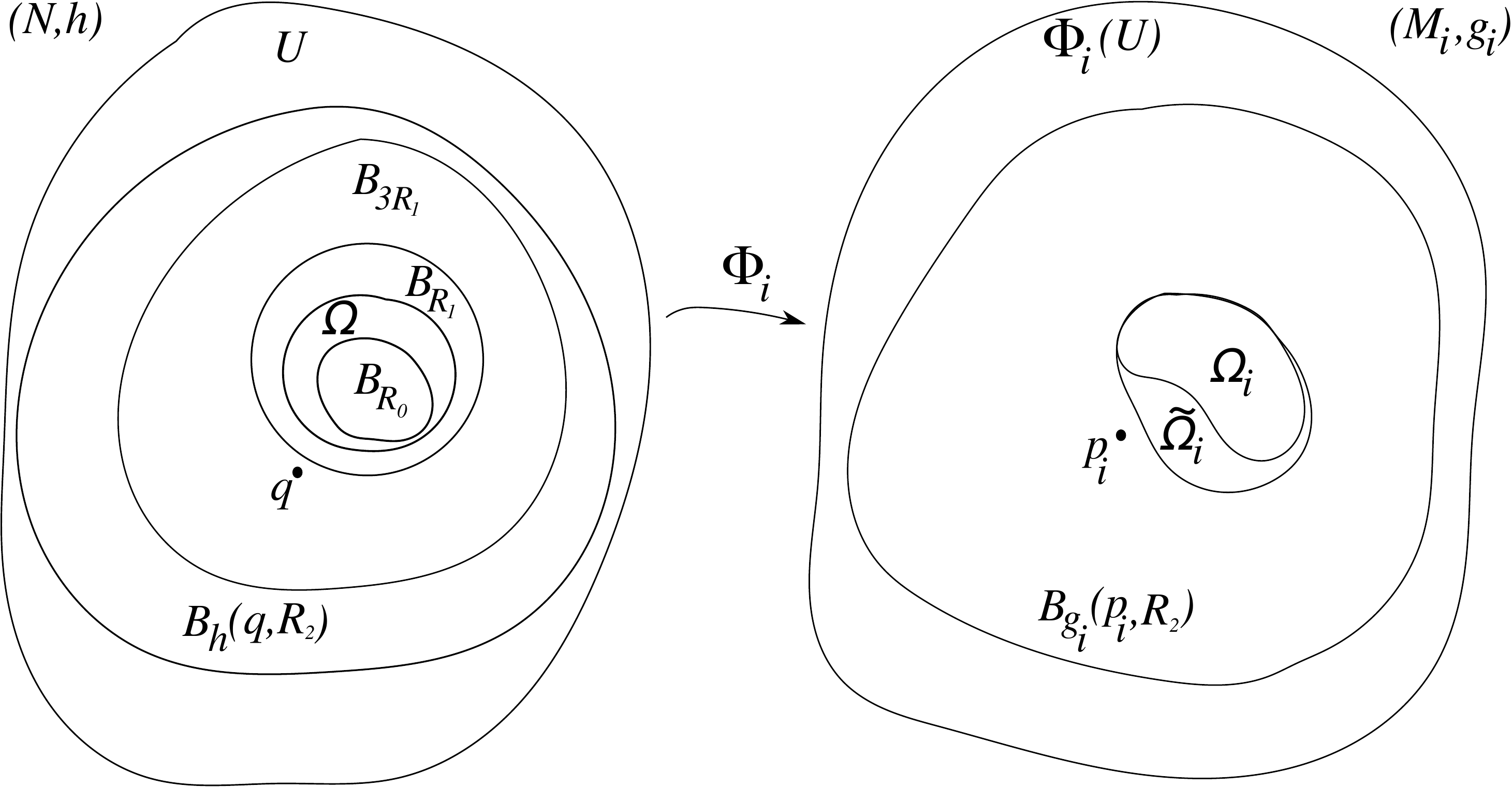}
\caption{\small The sets and maps used in the proof of Theorem \ref{thm:main}. Note $\Omega_i := \Phi_i(\Omega)$, and $\tilde \Omega_i$ is the minimizing hull of $\Omega_i$. The containment $\tilde\Omega_i \subset \Phi_i(U)$ is justified in Lemma \ref{lemma:tentacle}. 
 \label{fig_lsc}} 
 \end{center}
\end{figure}

Let $h_i := \Phi_i^* g_i$ on $U$ for $i \geq i_0$, so that $\Phi_i: (U,h_i) \longrightarrow (M_i,g_i)$ is trivially an isometry onto its image. The Cheeger--Gromov convergence means that $h_i \to h$ uniformly on $U$ (away from any fill-in regions). Thus, for $i \geq$ some $i_1 \geq i_0$, the isoperimetric ratio of $\Omega$ with respect to $h_i$ is at most $8\sqrt{\pi}$ (by (iii) above), the ratio of areas as measured by $h$ and $h_i$ are at most 2 on $U$, and also
\begin{equation}
\label{eqn_miso_K}
|\miso(\Omega,h_i) - \miso(\Omega,h)| < \epsilon.
\end{equation}

Since $\Phi_i$ is a smooth embedding, the set $\Omega_i:=\Phi_i(\Omega)$ is an allowable region in $M_i$ with smooth connected boundary $\partial \Omega_i$ (by (iii)), though it need not be outward-minimizing. Since we would like to apply Theorem \ref{thm:iso-adm2}, we consider the minimizing hull $\tilde\Omega_i$ of $\Omega_i$ in $(M_i^{\textrm{fill}}, g_i)$. 

Since $\Omega_i$ has smooth boundary, $\tilde \Omega_i$ is $C^{1,1}$ by Theorem \ref{regularity-thm}. We already have the upper bound $\mu$ for $\madm(M_i, g_i)$, as well as the upper bound $8\sqrt{\pi}$ for the isoperimetric ratio of $\tilde\Omega_i$ with respect to $g_i$.

The main issue is the lower bound for the isoperimetric constant of $(\tilde\Omega_i, g_i)$. But we control it using the following lemma, whose proof we postpone for the moment.

\begin{lemma} \label{lemma:tentacle}
Given the setup above, the closure of $\tilde \Omega_i$ is contained in $\Phi_i(U)$, for $i \geq i_1$.
\end{lemma}
It follows from the lemma that $c( \tilde\Omega_i, g_i)\ge c(\Phi_i(U),g_i) =c(U, h_i)$ for $i\geq i_1$. Since $h_i \to h$ uniformly on $U$, we know that $c(U, h_i)\ge \tfrac{1}{2} c(U,h)$ for $i \geq$ some $i_2 \geq i_1$. Finally, since $U\subset N$, we have $c(U,h)\ge c(N,h)$, giving us the desired uniform lower bound $c( \tilde\Omega_i, g_i)\ge  \tfrac{1}{2} c(N,h)$, for $i \geq i_2$.

In order to use Theorem \ref{thm:iso-adm2}, we also need to verify that the perimeter of $\tilde \Omega_i$ is sufficiently large. 
For $i\geq i_2$, (using Lemma \ref{lemma:tentacle} to guarantee $\partial \tilde \Omega_i$ is contained in the image of $\Phi_i$):
\begin{align}
|\partial\tilde\Omega_i|_{g_i} &= 
|\Phi_i^{-1}(\partial\tilde\Omega_i)|_{h_i}\nonumber\\
&\ge \frac{1}{2}|\Phi_i^{-1}(\partial\tilde\Omega_i)|_{h}\nonumber\\
&\ge \frac{1}{4}|\Phi_i^{-1}(\partial\tilde\Omega_i)|_{\delta}\nonumber\\
&\ge \frac{1}{4}|\partial B_{R_0}|_{\delta}\nonumber\\
&= \pi R_0^2. \label{eqn_area_lower_bound}
\end{align}
The first two inequalities follow from the bounds on the area ratios among $h_i$, $h$, and $\delta$. The last inequality uses the outward-minimizing property of spheres in Euclidean space, together with (ii), which says that $B_{R_0}\subset \Omega_i\subset \tilde{\Omega}_i$. In particular, by \eqref{eqn_area_lower_bound} and (\ref{eqn_pi_mu}), we see $|\partial\tilde\Omega_i|_{g_i} \geq 36\pi \mu^2$. 

Finally, note that the boundary of $\tilde \Omega_i$ does not touch the portion of $\partial M_i$ inside of $\Phi_i(U)$, since $\Omega \supset \partial N$ by (ii). Moreover, by Lemma \ref{lemma:tentacle}, $\tilde\Omega_i$ is allowable, and its boundary does not touch the portion of $\partial M_i$ outside $\Phi_i(U)$ either.

At last, we can apply Theorem \ref{thm:iso-adm2} to $\tilde \Omega_i$, noting that all hypotheses hold, for $i \geq i_2$:
\begin{equation}\label{important-inequality}
\miso( \tilde\Omega_i,g_i) < m_i + \frac{C}{\sqrt{|\partial  \tilde\Omega_i|_{g_i}}}.
\end{equation}

Feeding \eqref{eqn_area_lower_bound} into \eqref{important-inequality} and using \eqref{eqn_C_R_0}, we obtain
\begin{equation}
\label{eqn_m_i}
\miso( \tilde\Omega_i,g_i)< m_i + \frac{C}{\sqrt{\pi}R_0} < m_i +\epsilon.
\end{equation}
So for $i \geq i_2$, 
\begin{align*}
\miso(N,h) &< \miso(\Omega,h) + \epsilon &&\text{(by (i))}\\
&<\miso(\Omega,h_i) + 2\epsilon&&\text{(by (\ref{eqn_miso_K}))}\\
& =  \miso(\Omega_i,g_i) +2\epsilon&&\text{($\Phi_i$ is an isometry)}\\
&\le \miso(\tilde\Omega_i,g_i) +2\epsilon&&\text{(see below)}\\
&<m_i + 3\epsilon, &&\text{(by (\ref{eqn_m_i}))}
\end{align*}
where we used the fact that $\tilde\Omega_i$ has greater volume and less perimeter than $\Omega_i$ to compare their quasilocal isoperimetric masses with respect to $g_i$. (This comparison is only valid if $\miso(\Omega_i,g_i)\geq 0$. But if $\miso(\Omega_i,g_i)< 0$, then the last inequality above follows trivially, since $m_i \geq 0$.)  Since $\epsilon$ was arbitrary, the proof of Theorem \ref{thm:main} is complete (except for the proof of Lemma \ref{lemma:tentacle}), in the case that $\miso(N,h)<\infty$. If $\miso(N,h) = +\infty$, instead choose $\Omega$ in (i) so that $\miso(\Omega,h)>\epsilon^{-1}$. The rest of the argument is then identical, but it will conclude that $\liminf_{i \to \infty} m_i =+\infty$, a contradiction.

In order to prove Lemma \ref{lemma:tentacle}, we will use the following technical tool, stated in the language of integral currents.
\begin{definition} 
For any $\gamma\geq 1$, an integral current $S$ in $\R^n$ is \emph{$\gamma$-almost area-minimizing} if, for any ball $B$ with $B\cap\spt \partial S=\emptyset$ and any integral current $T$ with $\partial T=\partial(S\llcorner B)$, $|S\llcorner B|\leq\gamma |T|$.\footnote{The notation $S\llcorner B$ denotes the restriction of $S$ to $B$, which is just $S\cap B$ when $S$ is a submanifold.}
\end{definition}
The following result can be found, for example, in \cite[Lemma 5.1]{Bray-Lee:2009}.
\begin{lemma}\label{gamma}
Let $\gamma\geq 1$, and let $S$ be an $m$-dimensional $\gamma$-almost area-minimizing integral current in 
$\R^n$.  Let $x\in \spt S$, and let 
 $0<r<d(x,\spt\partial S)$.  Then
\[|S\llcorner B_r(x)|\geq \gamma^{1-m}\alpha_m r^m\]
where $\alpha_m$ is the area of the unit $m$-sphere.
\end{lemma}

\begin{proof}[Proof of Lemma \ref{lemma:tentacle}]
Suppose that the closure of $\tilde\Omega_i$ is not contained in $\Phi_i(U)$. Let $W_i$ be the fill-in region of $M_i$, and let 
 $\widetilde{\Omega_i\cup W_i}$ be the minimizing hull of $\Omega_i\cup W_i$.
Let $\Omega_i^\circ$ be the component of $\widetilde{\Omega_i\cup W_i}$ containing $\Omega_i$, and note that $\Omega_i^\circ$ is allowable. We claim that $\partial\Omega_i^\circ$ must be connected, and that it must touch $\partial\Omega_i$. To see the claim, suppose that $\partial\Omega_i^\circ$ has a component $\Sigma$ disjoint from $\partial\Omega_i$. 
By Theorem \ref{regularity-thm}, $\Sigma$ must be a smooth minimal surface except where it touches $\partial W_i=\partial M_i$. 
By the maximum principle and the fact that there are no minimal surfaces in the interior of $M_i$, it follows that $\Sigma$ must coincide with a component of $\partial W_i$. But this contradicts the connectedness of $\Omega_i^\circ$, proving the claim.

Thus $\partial \Omega_i^\circ$ is connected and touches $\partial \Omega_i$. Since $\Omega_i^\circ \supset \tilde \Omega_i$ (since the former is an outward-minimizing region that contains $\Omega_i$), we see that $\partial \Omega_i^\circ$ is not contained in $\Phi_i(U)$. See Figure~\ref {fig:tentacle} for an illustration of the the setup of this proof. In particular, $\partial \Omega_i^\circ$ must intersect $\Phi_i(\partial B_{R_1})$, $\Phi_i(\partial B_{2R_1})$, and $\Phi_i(\partial B_{3R_1})$. Pulling back to $N$, we see that
\[T_i := \Phi_i^{-1} (\partial \Omega_i^\circ \cap \Phi_i(\text{interior}(B_{3R_1} \smallsetminus B_{R_1})) )\]
 is an area-minimizing surface with respect to $h_i$, without boundary, in $\text{interior}(B_{3R_1} \smallsetminus B_{R_1})$. Moreover, there exists a point $x_i \in T_i \cap \partial B_{2R_1}$. 

\begin{figure}[ht]
\begin{center}
\includegraphics[scale=0.60]{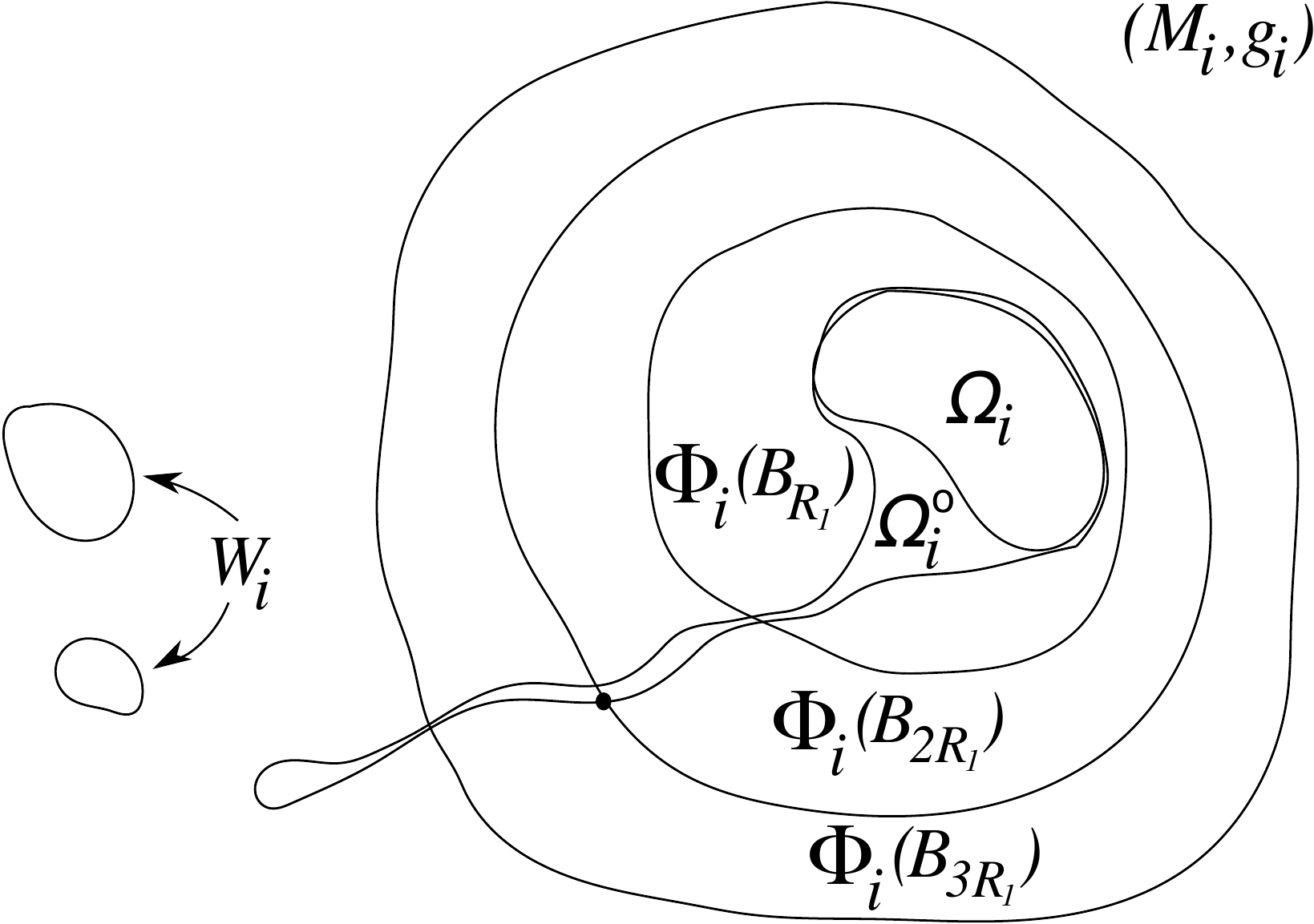}
\caption{\small The setup of the proof of Lemma \ref{lemma:tentacle}. This lemma shows that $\Omega_i^\circ$ cannot actually escape far out into the asymptotically flat end as shown.  \label{fig:tentacle}}
\end{center}
\end{figure}

Since $i \geq i_1$, we know that the ratios of areas measured by $h_i$ and $h$ on $U$ are at most $2$. Then by item (iv), the ratio of areas measured by $h_i$ and $\delta$ is at most $4$, and therefore $T_i$ is $16$-almost area-minimizing in Euclidean $\R^3$. 
Therefore  
\begin{align*}
|\partial \Omega|_h + 8\pi\mu^2
& \ge \frac{1}{2}|\partial\Omega|_{h_i} + 8\pi m_i^2
 &&\text{(comparing area ratio between $h$ and $h_i$)}\\
& \ge \frac{1}{2}|\partial\Omega|_{h_i} + \frac{1}{2}|\partial M_i|_{g_i} 
 &&\text{(Bray's version of the Penrose inequality, Theorem \ref{thm:RPI})}\\
& \ge \frac{1}{2}\left|\partial{(\Omega_i\cup W_i)} \right|_{g_i} &&\text{(since $\Phi_i$ is isometry, and by definition of $W_i$)}\\
& \ge \frac{1}{2}\left|\partial\widetilde {(\Omega_i\cup W_i)} \right|_{g_i} &&\text{(by definition of minimizing hull)}\\
& \ge \frac{1}{2}|\partial\Omega_i^\circ|_{g_i} &&\text{(by definition of $\Omega_i^\circ$)}\\
&\ge \frac{1}{2}|T_i \cap B(x_i,R_1)  |_{h_i} &&\text{($T_i \cap B(x_i,R_1)$ is isometric to  a subset of $\partial \Omega_i^\circ$)}\\
&\ge \frac{1}{8}|T_i \cap B(x_i,R_1)  |_{\delta} &&\text{(comparing area ratio between $h_i$ and $\delta$)}\\
&\ge \frac{1}{8} \cdot\frac{1}{16} 4\pi R_1^2,
\end{align*}
where we applied Lemma \ref{gamma} on the last line. But this contradicts inequality \eqref{size-R}.
\end{proof}

A similar technique was used in \cite{Jauregui} to rule out ``tentacles'' of a minimizing hull extending far out into an asymptotically flat end.

\section*{Appendix: Equivalence of definitions of isoperimetric mass}

Recall Definition \ref{def:isop_mass}, in which $\miso(M,g)$ is defined. The following result is never used in the paper, but it is an interesting fact about isoperimetric mass.

\begin{lemma}
Let $(M,g)$ be a $C^0$ asymptotically flat 3-manifold. We define an alternate version of isoperimetric mass as follows:
\[\widetilde\miso(M,g) = \sup_{\{\Omega_i\}_{i=1}^\infty} \left(\limsup_{i \to \infty} \miso(\Omega_i,g)\right),\]
where the supremum is taken over all sequences $\{ \Omega_i \}_{i=1}^\infty$ of allowable regions such that 
$|\partial^*\Omega_i| \to\infty$ as $i\to\infty$. If $\widetilde\miso(M,g) > 0$, then 
\[ \widetilde\miso(M,g) =\miso(M,g). \]
In other words, defining the isoperimetric mass using exhaustions is equivalent to using sequences whose perimeters become arbitrarily large, when the latter is positive.
 \end{lemma}
\begin{proof}
We only need to prove that $\widetilde\miso(M,g) \le\miso(M,g)$ since the other inequality is immediate.
 Let $W$ be any allowable region that contains $\partial M$. 
Let $\{ \Omega_i \}_{i=1}^\infty$ be a sequence of allowable regions  such that $|\partial^*\Omega_i| \to\infty$ as $i\to\infty$, and such that $\miso(\Omega_i)>0$ for all $i$ (which can be found because $\widetilde\miso(M,g) > 0$). Define ${\Omega}_i':=  W \cup \Omega_i$, which are allowable regions. We will prove that 
\begin{equation}\label{goal}
 \limsup_{i\to\infty} \miso(\Omega_i') \ge \limsup_{i\to\infty} \miso(\Omega_i).
 \end{equation}

If it happens that $|\partial^*\Omega_i'| \le |\partial^*\Omega_i|$, then we can see that
\begin{align*}
\miso(\Omega_i') &= \frac{2}{|\partial^*\Omega_i'|} \left(|\Omega_i' | - \frac{1}{6\sqrt{\pi}} |\partial^*\Omega_i' |^{3/2}\right)\\
&\ge \frac{2}{|\partial^*\Omega_i' |} \left(|\Omega_i | - \frac{1}{6\sqrt{\pi}} |\partial^*\Omega_i |^{3/2}\right)\\
&\ge \frac{2}{|\partial^*\Omega_i|} \left(|\Omega_i | - \frac{1}{6\sqrt{\pi}} |\partial^*\Omega_i |^{3/2}\right)\\
&=\miso(\Omega_i)
\end{align*}
where we used positivity of $\miso(\Omega_i)$ in the second inequality. In the case when $|\partial^*\Omega_i'| > |\partial^*\Omega_i|$, we estimate
\begin{align*}
\miso(\Omega_i') &= \frac{2}{|\partial^*\Omega_i'|} \left(|\Omega_i' | - \frac{1}{6\sqrt{\pi}} |\partial^*\Omega_i' |^{3/2}\right)\\
&\ge \frac{2}{|\partial^*\Omega_i' |}\left(|\Omega_i | - \frac{1}{6\sqrt{\pi}}( |\partial^*\Omega_i| + |\partial^*W|)^{3/2}\right)\\
&\ge \frac{2}{|\partial^*\Omega_i' |}\left(|\Omega_i | - \frac{1}{6\sqrt{\pi}} |\partial^*\Omega_i|^{3/2} - O(|\partial^*\Omega_i|^{1/2})\right)\\
&= \frac{2}{|\partial^*\Omega_i' |}\left(|\Omega_i | - \frac{1}{6\sqrt{\pi}} |\partial^*\Omega_i|^{3/2}\right) - \frac{1}{|\partial^*\Omega_i' |} O(|\partial^*\Omega_i|^{1/2})\\
&\ge \frac{2}{|\partial^*\Omega_i | +|\partial^*W|}\left(|\Omega_i | - \frac{1}{6\sqrt{\pi}} |\partial^*\Omega_i|^{3/2}\right) -  O(|\partial^*\Omega_i|^{-1/2})\\
\end{align*}
where we used positivity of $\miso(\Omega_i)$ and the bounds $|\partial^*\Omega_i'| \le |\partial^*\Omega_i| +|\partial^* W|$ and $|\partial^*\Omega_i'| > |\partial^*\Omega_i|$ in the last line. Here, ``big $O$'' depends on $W$ but not on $i$.
Since ${(|\partial^*\Omega_i | +|\partial^*W|)^{-1}} = |\partial^*\Omega_i |^{-1} +  O(|\partial^*\Omega_i |^{-2})$, we obtain
\begin{align*}
\miso(\Omega_i') &\ge \miso(\Omega_i) -  O(|\partial^*\Omega_i|^{-1/2}) - O(|\Omega_i|\cdot |\partial^*\Omega_i|^{-2}) \\
&= \miso(\Omega_i) -  O(|\partial^*\Omega_i|^{-1/2})
\end{align*}
where the last line follows from the isoperimetric inequality for $(M,g)$. Inequality \eqref{goal} now follows. From this inequality, we conclude that as long as $\widetilde\miso(M,g)>0$, it can be computed using only sequences of regions that each contain $W$. The result now follows from a straightforward diagonalization argument, considering a sequence of sets $W$ exhausting $M$.
\end{proof}

\end{document}